\documentclass[11pt]{article}

\usepackage[a4paper, left=3cm,top=2.5cm,right=3cm,bottom=3.5cm]{geometry}
\usepackage{cmbright}
\usepackage[T1]{fontenc}

\usepackage{graphicx}
\usepackage{mathtools}
\usepackage{psfrag}
\usepackage{amssymb}
\usepackage{amsmath}
\usepackage{authblk}
\usepackage{siunitx}

\usepackage{amsthm}
\usepackage{hyperref}
\usepackage{todonotes}

\usepackage{algorithm}
\usepackage{algorithmic}
\usepackage{subfig}

\usepackage{xspace}

\usepackage{enumitem}
\setitemize{label=\scriptsize{$\blacksquare$},topsep=0em}
\setenumerate{label=(\alph*), topsep=0em}

\newtheorem{theorem}{Theorem}

\newtheorem{proposition}[theorem]{Proposition}

\newtheorem{remark}[theorem]{Remark}
\newtheorem{example}[theorem]{Example}

\theoremstyle{definition}
\newtheorem{condition}[theorem]{Condition}
\newtheorem{definition}[theorem]{Definition}

\newcommand{\ph}{\varphi}

\newcommand{\R}{\mathbb R}

\newcommand{\N}{\mathbb N}
\newcommand{\sph}{\mathbb S}
\newcommand{\edot}{\,\cdot\,}
\newcommand{\trans}{\mathsf{T}}

\DeclareMathOperator*{\argmin}{arg\,min}

\newcommand{\la}{\lambda}

\newcommand{\ran}{\operatorname{Ran}}

\newcommand\abs[1]{\left\vert#1\right\vert}
\newcommand\sabs[1]{{\lvert#1\rvert}}
\newcommand\norm[1]{{\left\Vert#1\right\Vert}}
\newcommand\snorm[1]{\Vert#1\Vert}
\newcommand{\enorm}{\left\|\;\cdot\;\right\|}
\newcommand\set[1]{{\left\{#1\right\}}}
\newcommand{\kl}[1]{\left(#1\right)}

\newcommand\inner[2]{\left\langle#1,#2\right\rangle}
\newcommand\sinner[2]{\langle#1,#2\rangle}

\newcommand{\cm}{}

\makeatletter
\newcommand*\bigcdot{\mathpalette\bigcdot@{.6}}
\newcommand*\bigcdot@[2]{\mathbin{\vcenter{\hbox{\scalebox{#2}{$\m@th#1\bullet$}}}}}
\makeatother

\newcommand{\La}{\Lambda}
\newcommand{\bias}{b}
\newcommand{\signal}{x}
\newcommand{\zsignal}{z}

\newcommand{\signalP}{x_\Plus}

\newcommand{\source}{p_0}
\newcommand{\data}{y}
\newcommand{\noise}{\xi}

\newcommand{\XXX}{\mathbb X}
\newcommand{\XX}{X}
\newcommand{\YY}{Y}

\newcommand{\DD}{D}
\newcommand{\WW}{\mathcal V}

\newcommand{\modt}{\nu}

\newcommand{\wave}{\mathcal{F}}
\newcommand{\BP}{\mathbf{B}}
\newcommand{\samp}{\mathbf{S}}
\newcommand{\Fo}{\mathbf{F}}
\newcommand{\Ko}{K}

\newcommand{\nlo}{\sigma}
 \newcommand{\Ao}{\mathbb  A}
\newcommand{\Wo}{\mathbb V}
\newcommand{\WWe}{\mathbb  V}
\newcommand{\WWd}{\mathbb  W}

\newcommand{\breg}{\mathcal{B}}
\newcommand{\err}{\mathcal{E}}
\newcommand{\simm}{\mathcal{D}}
\newcommand{\reg}{\mathcal{R}}    \newcommand{\tik}{\mathcal{T}}
\newcommand{\fun}{\mathcal{F}}

\newcommand{\nlf}{\psi}
\newcommand{\NN}{\mathbold{\Phi}}
\newcommand{\NNe}{\mathbold{\Phi}}
\newcommand{\NNd}{\mathbold{\Psi}}

\newcommand{\sgn}{\operatorname{sign}}
\newcommand{\bgO}{\mathcal{O}}

\newcommand{\xxt}{x}
\newcommand{\zzt}{z}
\newcommand{\rrt}{r}
\newcommand{\Ntrain}{N}
\newcommand{\ntrain}{n}

\usepackage{soul}
\usepackage{xcolor}
\colorlet{lred}{red!40}
\colorlet{lgreen}{green!40}
\colorlet{lblue}{blue!40}
\colorlet{lviolet}{violet!40}

\def\Plus{\texttt{+}}

\numberwithin{equation}{section}
\numberwithin{figure}{section}
\numberwithin{theorem}{section}

\allowdisplaybreaks

\title{NETT: Solving  Inverse Problems with Deep Neural Networks}

\author[1]{Housen~Li}
\author[2]{Johannes~Schwab}
\author[2]{Stephan~Antholzer}
\author[2,$\star$]{Markus~Haltmeier}

\affil[1]{Institute for Mathematical Stochastics, University of G\"ottingen,
Goldschmidtstrasse 7, 37077 G\"ottingen, Germany\vspace{1em}}

\affil[2]{Department of Mathematics, University of Innsbruck\authorcr
Technikerstrasse 13, 6020 Innsbruck, Austria\vspace{1em}}

\affil[$\star$]{Correspondence: {\tt markus.haltmeier@uibk.ac.at}}

\date{ }

\begin{document}

\maketitle

\begin{abstract}
Recovering a function or high-dimensional parameter vector from indirect measurements is a central task in various scientific areas. Several methods for solving such inverse problems are well developed and well understood. Recently, novel algorithms using deep learning and neural networks for inverse problems appeared. While still in their infancy, these techniques show astonishing performance for applications like low-dose CT or various sparse data problems.
	However, {\cm there are few} theoretical results for deep learning in inverse problems.
In this paper, we establish  {\cm a complete convergence analysis} for the proposed NETT (Network Tikhonov) approach to inverse problems. NETT considers data consistent solutions having small value of a regularizer defined by a trained neural network.
    We derive well-posedness results and quantitative error estimates, and propose a possible strategy for training the regularizer. {\cm Our theoretical results and framework are different from any previous work using neural networks for solving inverse problems.}
A possible data driven regularizer is proposed. Numerical results are presented for a tomographic sparse data problem, which demonstrate good performance of NETT even for unknowns of different type from the training data.
{\cm To derive the convergence and convergence rates results we introduce a new framework based on the absolute Bregman distance  generalizing the  standard Bregman distance from the convex to the non-convex case.}


\end{abstract}

\section{Introduction}
\label{sec:intro}

We study the stable solution of inverse problems of the form
\begin{equation}\label{eq:ip}
	\text{Estimate $ \signal \in \DD $ from data } \quad
	\data_\delta = \Fo( \signal )   + \noise_\delta    \,.
\end{equation}
Here $\Fo \colon \DD  \subseteq \XX \to \YY$  is a  possibly non-linear operator
between reflexive Banach spaces $(\XX, \enorm)$ and $(\YY, \enorm)$ with domain $\DD$. We allow a possibly  infinite-dimensional
function space  setting, but clearly
the approach and results apply to a finite dimensional setting as well.
The element $\noise_\delta \in \YY$ models the unknown data error (noise)
which is assumed to satisfy the estimate   $\snorm{\noise_\delta } \leq \delta $
for some  noise level $\delta  \geq 0$.   We focus on the ill-posed
(or ill-conditioned) case where without additional  information,
the solution of \eqref{eq:ip} is either highly unstable,
highly underdetermined, or both.
Many inverse problems in biomedical imaging,
geophysics, engineering sciences, or elsewhere  can be written in such a form (see, for example,  \cite{engl1996regularization,natterer01mathematical,scherzer2009variational}).
For its stable solution one has to employ regularization methods,  which are based on
approximating~\eqref{eq:ip} by neighboring well-posed problems that enforce stability and  uniqueness.

\subsection{NETT regularization}

Any method  for the stable solution of \eqref{eq:ip} uses, either implicitly or explicitly,
a-priori information about  the unknowns to be recovered. Such
information can be that  $\signal$ belongs to a certain  set
of admissible elements    or that $\signal$ has small value of a regularizer (or regularization functional) $\reg \colon \XX \to [0, \infty]$.  In this paper we focus on the latter situation, and assume that   the regularizer takes the  form
\begin{equation} 	\label{eq:cnn}
\forall \signal \in \XX \colon
	\quad
	\reg(\signal)
	=
	\reg(\Wo,\signal)
	\coloneqq
	\nlf(\NN(\Wo,\signal)) \,.
\end{equation}
Here $\nlf \colon  \XXX_L \to [0, \infty]$  is a scalar
functional and $\NN(\Wo, \edot) \colon \XX \to \XXX_L$   a  neural  network of depth $L$ where
$\Wo \in \WW$, for some vector space $\WW$, contains free parameters that can be adjusted to available  training data (see Section~\ref{sec:framework}  for a precise formulation).

With the regularizer \eqref{eq:cnn}, we approach~\eqref{eq:ip} via
\begin{equation}\label{eq:nett}
 \tik_{\alpha;\data_\delta}(\signal)\coloneqq\simm(\Fo (\signal), \data_\delta)
 + \alpha \reg(\Wo,\signal) \to \min_{\signal \in \DD} \,,
\end{equation}
where  $\simm \colon \YY \times \YY \to [0, \infty] $  is an appropriate
similarity measure in the data space enforcing data consistency.
One may take $\simm(\Fo (\signal), \data_\delta) = \snorm{\Fo (\signal) - \data_\delta}^2$  but  also other similarity  measures such as the Kullback-Leibler divergence
 (which, among others, is used    in emission tomography)
 are reasonable choices.
Optimization problem  \eqref{eq:nett} can  be seen as a particular instance of generalized
Tikhonov regularization for solving \eqref{eq:ip} with a neural network as regularizer.
We therefore name  \eqref{eq:nett}   network Tikhonov (NETT)
approach for  inverse problems.

In this paper, we show that under reasonable assumptions, the NETT approach \eqref{eq:nett}
is stably solvable.  As $\delta \to 0$, the regularized solutions
 $ \signal_{\alpha,\delta}  \in \argmin_{\signal} \tik_{\alpha;\data_\delta}(\signal)   $
are shown to  converge   to $\reg(\Wo, \edot)$-minimizing solutions of  $\Fo( \signal )   =  \data_0$.
 Here and below $\reg(\Wo, \edot)$-minimizing solutions of $\Fo( \signal )   =  \data_0$
 are defined as any element
 \begin{equation}\label{eq:RminW}
	\signalP \in \argmin\set{\reg(\Wo,\signal) \mid \signal \in \DD \wedge
\Fo( \signal )   =  \data_0 } \,.
\end{equation}
Additionally, we derive convergence rates (quantitative error estimates)
between   $\reg(\Wo, \edot)$-minimizing solutions $\signalP$ and regularized solutions $\signal_{\alpha,\delta}$.
As a consequence,  \eqref{eq:nett} provides a  stable solution scheme
for \eqref{eq:ip} using  data consistency and encoding a-priori knowledge via neural networks.
For proving norm convergence and convergence rates, we introduce the absolute Bregman distance as a new generalization of the  standard Bregman distance for non-convex regularization.

\subsection{Possible regularizers}

The network regularizer $\reg(\Wo, \edot)$ can either be user-specified,
or a  trained network, where  free parameters are adjusted on appropriate
training data. Some examples are as follows.

\begin{itemize}[wide]
\item \textbf{\cm Non-linear $\ell^q$-regularizer:}
A simple user-specified instance of the regularizer \eqref{eq:cnn} is  the
 convex $\ell^q$-regularizer $\reg(\Wo,\signal) = \sum_{\la \in \La}
v_\la \abs{ \inner{\signal}{\ph_\la}}^q$. Here $(\ph_\la)_{\la \in \La}$ is a prescribed
basis or frame  and $(v_\la)_{\la \in \La}$ are weights.
In this case, the neural  network is simply given by the analysis operator
$\NN(\Wo, \edot) \colon \XX \to \ell^2(\La) \colon
\signal \mapsto \inner{\signal}{\ph_\la} $ and  NETT regularization
reduces to sparse $\ell^q$-regularization \cite{daubechies2004iterative,grasmair2008sparse,grasmair2011necessary,lorenz2008convergence,ramlau2006tikhinov}.
This form of the  regularizer can also be combined with a training procedure
by adjusting the weights $(v_\la)_{\la \in \La}$ to a class of training data.

In this paper  we study  a {\cm non-linear extension of
$\ell^q$-regularization}, where the {\cm (in general) non-convex network regularizer}  takes the form
\begin{equation}\label{eq:rq}
 \reg(\Wo,\signal) = \sum_{\la \in \La}
   v_\la \abs{ \NN_\la(\Wo,\signal)}^q \,,
\end{equation}
with  $q \geq 1$ and $\NN(\Wo, \edot) = (\NN_\la(\Wo, \edot))_{\la \in \La}$ being a possible
non-linear neural network with multiple layers. In Section~\ref{sec:ellq} we present convergence results
for this {\cm non-linear} generalization  of  $\ell^q$-regularization.
	{\cm By selecting non-negative weights, one can easily construct networks that are convex
with respect to the inputs~\cite{amos2017input}. In this work, however, we consider the general situation
of arbitrary weights, in which the network regularizer~\eqref{eq:rq} can be non-convex.}

\item \textbf{CNN  regularizer:}
The  network regularizer  $\reg(\Wo, \edot)$  in \eqref{eq:cnn}  may also be defined by  a
convolutional neural network (CNN) $\NN(\Wo, \edot)$,
containing   free parameters  that can be  adjusted on
	appropriate training data.  The CNN  can be  trained in such  a way, that the regularizer has small  value for  elements $\signal$ in a {\cm set of training phantoms}
	and larger  value  on a class of  un-desirable phantoms.
{\cm The class of un-desirable phantoms can be elements
containing undersampling artifacts, noise, or both}.  In Section~\ref{sec:num},
we present a possible regularizer design  together with a strategy for
training  the CNN {\cm to remove undersampling artifacts}. We present
numerical results demonstrating that our approach performs well in practice for a sparse tomographic data problem.
\end{itemize}

\subsection{Comparison to previous work}

Very recently, several deep learning  approaches for  inverse problems have been developed (see\footnote{We initially submitted our paper a recognized journal in February 28, 2018. On June 18, 2019, we have been informed that the paper is rejected. Since so much work has been done in the emerging field of deep learning in inverse problems, for the present version, we did not update the reference with all interesting papers, but only closely related work.}, for example, \cite{adler2017solving,antholzer2017deep,chen2017lowdose,han2016deep,jin2017deep,kelly2017deep,kofler2018u,schlemper2017deep,wang2016perspective,wang2016accelerating,wurfl2016deep,zhang2016image}).
In all these approaches, a trained network $\NN_{\rm rec}(\Wo, \edot) \colon \YY \to \XX$  maps measured data to the desired output image.

{\cm Two-step reconstruction networks} take the form $\NN_{\rm rec}(\Wo, \edot) = \NN_{\rm CNN}(\Wo, \edot) \circ \BP$, where  $\BP \colon \YY \to \XX$ maps the data to the reconstruction space (backprojection; no free parameters)  and $\NN_{\rm CNN}(\Wo, \edot)    \colon \XX \to \XX$ is a neural network, for example a convolutional neural network (CNN),  whose free parameters are  adjusted to the training data. This basic form allows the use of well established CNNs for image reconstruction  \cite{goodfellow2016deep} and
already demonstrates impressing results.
Network cascades~\cite{kofler2018u,schlemper2017deep} and trained iterative schemes~\cite{adler2017solving,aggarwal2018modl,hammernik2018learning,rick2017one}
learn free parameters  in iterative schemes.
In  such approaches, the reconstruction network  can be written
in the form
\begin{equation*}
\NN_{\rm rec}(\Wo, \data)
= \kl{ \NN_N (\Wo_N, \edot)
\circ  \BP_N(\data, \edot)
\circ \cdots  \circ \NN_1 (\Wo_1, \edot)  \circ  \BP_1 (\data, \edot)}(\signal_0) \,,
\end{equation*}
where  $\signal_0$ is the initial guess,
$\NN_{k}(\Wo_k, \edot)    \colon \XX \to \XX$ are CNNs  that can be
trained,  and  $\BP_k(\data, \edot) \colon \XX \to \XX $  are iterative
updates  based on the forward operator and the data.
The iterative updates may be defined by a
gradient step with respect to the given inverse problem. The  free parameters
are adjusted to available  training  data.

Network cascades and trained iterative schemes repeatedly make use
of the  forward problem  which might yield increased data
consistency compared to  the first class of methods.
Nevertheless,  in existing approaches, no provable non-trivial
estimates bounding the data consistency term
$\simm (\Fo(\signal) ,\data)$
are available; data consistency  can only be  guaranteed for the training data
$(\Fo(\zzt_\ntrain), \zzt_\ntrain )_{\ntrain=1}^{\Ntrain}$ for  which  the parameters
in the neural network are optimized.
 This may results in instability and  degraded reconstruction
 quality if the unknown  to be recovered is not similar  enough to the
 class of employed training data.
The proposed NETT bounds  the data consistency term
$\simm (\Fo(\signal_{\alpha, \delta}) ,\data)$ also for data
outside the training set.    We expect the  combination of the forward problem
and a  neural network via  \eqref{eq:nett} (or, for the noiseless case, \eqref{eq:RminW}) to increase
reconstruction quality, especially  in the case of limited access
to a large amount  of appropriate training data.

Note, further, that the formulation of NETT~\eqref{eq:nett} separates the noise characteristic and the a-priori information of unknowns. This allows us to incorporate the knowledge of data generating mechanism, e.g. Poisson noise or Gaussian noise, by choosing the corresponding log-likelihood as the data consistency term, and also simplifies the training process of $\reg(\Wo, \edot)$, as it to some extend avoids the impact of noise. Meanwhile, this enhances the interpretability of the resulting approach: we on the one hand require its fidelity to the data, and on the other penalize unfavorable features (e.g. artifacts in tomography).

{\cm
    An early related work~\cite{romano2017little} uses denoisers as a regularization term which also includes certain CNNs. In~\cite{aggarwal2018modl}, they use a residual network for $\NN$ and $\nlf(\cdot)=\norm{\edot}_2^2$.
    Another related work~\cite{rick2017one} uses a learned proximal operator instead of a regularization term.
    After the present paper was initially submitted, other works explored the idea of neural networks as regularizers.
    In particular, in \cite{Lunz2018} a regularizer has been proposed that distinguishes the distributions of desired images and noisy images. We note that neither convergence nor convergence rates results have been derived by any
    work using neural networks as regularizer.}

{\cm
The results in this paper are a main step for the regularization of inversion problems with neural networks.
For the first time, we present a complete convergence analysis and derive convergence rate under reasonable
assumptions.}  Many additional issues can be addressed in future work.
This  includes the design of appropriate CNN regularizers,
the development of efficient algorithms for minimizing \eqref{eq:nett},  and the consideration of
other regularization strategies for \eqref{eq:RminW}.  The  focus of the present
paper is on the theoretical  analysis of NETT and demonstrating the
feasibility  of our approach;  detailed  comparison with other methods
in terms of reconstruction quality, computational performance  and
applicability  to real-world  data is beyond our scope here and will be addressed in future work.

\subsection{Outline}

The rest of this paper is organized as follows.
In Section~\ref{sec:nett}, we describe the proposed NETT framework for solving inverse problems. We  show  its stability and derive convergence in the weak topology
(see  Theorem~\ref{thm:well}).    To  obtain  the strong convergence of
NETT, we  introduce a new notion of total  non-linearity  of  non-convex functionals.
For totally non-linear regularizers, we  show norm convergence of NETT
(see Theorem~\ref{thm:strong}).   Convergence rates (quantitative error estimates)  for NETT are derived in Section~\ref{sec:rates}.  Among others, we
derive  a convergence rate result in terms of the absolute Bregman
distance (see Proposition \ref{prop:cs}).   A framework  for  learning the data driven regularizer is
proposed in Section~\ref{sec:auto}, and applied to a sparse data problem in photoacoustic  tomography in Section~\ref{sec:num}.    The paper concludes with a short summary and outlook presented in
Section~\ref{sec:conclusion}.

\section{NETT regularization}
\label{sec:nett}

In the section we introduce the  proposed NETT
and analyze its well-posedness
(existence, stability and weak convergence). We introduce the novel concepts of  absolute Bregman distance and
total non-linearity, which are applied to  establish convergence of NETT with respect to  the norm.

\subsection{The NETT framework}
\label{sec:framework}

Our goal is to solve \eqref{eq:ip} with $\norm{\noise_\delta} \leq \delta$ and $\delta>0$.  For that
purpose we consider minimizing the NETT functional \eqref{eq:nett},
where the regularizer $\reg(\Wo, \edot) \colon \XX  \to [0, \infty]$ in \eqref{eq:cnn} is defined by a neural network of the form
\begin{equation} \label{eq:cnn2}
	\NN(\Wo,\signal)
	\coloneqq
	(\nlo_{L}  \circ \Wo_L \circ \nlo_{L-1} \circ \Wo_{L-1} \circ  \cdots \circ \nlo_1 \circ \Wo_1) (\signal) \,.
\end{equation}
Here $L$ is the depth of the network (the number of layers after the input layer) and
$\Wo_\ell (\signal)= \Ao_\ell (\signal) +  \bias_\ell$
are  affine linear operators between Banach spaces $\XXX_{\ell-1}$ and $\XXX_{\ell-1/2}$;
we take  $\XXX_0  \coloneqq  \XX $. The operators $\Ao_\ell \colon  \XXX_{\ell-1} \to \XXX_{\ell-1/2}$
are the linear parts and  $\bias_\ell \in \XXX_{\ell-1/2}$  the  so-called bias terms. The operators   $\nlo_\ell \colon \XXX_{\ell-1/2} \to \XXX_\ell$ are possibly non-linear and the functionals $\nlf \colon \XXX_L \to [0, \infty]$  are possibly  non-convex. Note that we use two different  spaces $\XXX_{\ell-1}$ and $\XXX_{\ell-1/2}$
in each  layer  because common operations in networks
like max-pooling, downsampling or upsampling change the domain space.

As common in machine  learning, the affine mappings $\Wo_\ell$  depend on free parameters    that can be adjusted in the training phase, whereas  the non-linearities   $\nlo_\ell$ are fixed. Therefore $\Wo_\ell $ and $\nlo_\ell$ are treated separately
and only the affine  part $\Wo =(\Wo_\ell)_{\ell=1}^L $ is indicated in the notion of
the neural network regularizer $\reg(\Wo, \edot)$.
Throughout our  theoretical analysis  we assume
$\reg(\Wo, \edot)$ to be given and all free parameters to be trained before the
minimization of~\eqref{eq:nett}. In Section~\ref{sec:auto}, we present a possible framework
for  training a neural network regularizer.

\begin{remark}[CNNs in Banach space setting]
A\label{rem:cnn} typical instance for the neural network in NETT \eqref{eq:cnn},
is a deep convolutional neural network (CNN). In a possible infinite dimensional setting,  such CNNs can be written in the form \eqref{eq:cnn2},
where the involved spaces satisfy   $\XXX_\ell  \coloneqq  \ell^p(\Lambda_\ell,  \XX_\ell)$ and
$\XXX_{\ell-1/2}  \coloneqq  \ell^p(\Lambda_\ell,  \XX_{\ell-1/2})$ with $p \geq 1$,
$\XX_\ell$ and $\XX_{\ell-1/2}$ being function spaces, and $\Lambda_\ell$
being an at most countable set that specifies the number of different filters
(depth) of the $\ell$-th layer.
The   linear operators $\Ao_\ell \colon \ell^p(\Lambda_{\ell-1},  \XX_{\ell-1})
\to \ell^p(\Lambda_\ell,  \XX_{\ell-1/2})$ are taken as
\begin{equation} \label{eq:nccc}
\forall \signal \in \ell^p(\Lambda_{\ell-1},  \XX_{\ell-1})
\; \forall \ell  \in \Lambda_\ell
\colon \quad
\Ao_\ell (\signal) =
\kl{ \sum_{\mu \in \La_{\ell-1}} \Ko_{\la, \mu}^{(\ell)} (\signal_\mu)}_{\la \in \Lambda_\ell}  \,,
\end{equation}
where $\Ko_{\la, \mu}^{(\ell)}   \colon  \XX_{\ell-1} \to \XX_{\ell-1/2}$ are
convolution operators.

We point out, that in the existing machine learning literature,
only finite dimensional settings have been considered so far, where
$ \XXX_\ell$ and $ \XXX_{\ell-1/2}$ are finite dimensional spaces. In such a finite dimensional setting,
we can take  $\XX_\ell=  \R^{N_1^{\ell} \times N_2^{\ell}}$,   and   $\Lambda_\ell$
as a set with $N_c^{\ell}$ elements.
 One then can identify $ \XXX_\ell = \ell^p(\Lambda_\ell,  \XX_\ell) \simeq \R^{N_1^{\ell} \times N_2^{\ell} \times N_c^{\ell} }$
 and interpreted its elements  as stack of discrete images (the same holds for $ \XXX_{\ell-1/2}$).
In typical  CNNs, either the dimensions $N_1^{\ell} \times N_2^{\ell}$ of the base space
$\XX_\ell$ are progressively reduced  and number of channels $N^{\ell}_c$ increased,  or
vice versa.
While we are not aware of any infinite dimensional general formulation of CNNs,
our proposed  formulation \eqref{eq:cnn2}, \eqref{eq:nccc} is the natural infinite-dimensional
Banach space version of  CNNs, which  reduces to standard CNNs \cite{goodfellow2016deep}
in the finite  dimensional setting.
\end{remark}

Basic convex  regularizers
are sparse $\ell^q$-penalties $\reg(\Wo,\signal)  =  \sum_{\la \in \La} v_\la \abs{ \inner{\ph_\la}{\signal}}^q$. In this case  one may take  \eqref{eq:cnn2} as a single-layer neural network
with $\XXX_1 = \ell^2( \Lambda_1, \R)$,
$\nlo = \operatorname{Id}$ and $\NN(\Wo,\signal) = \Wo (\signal) = (\inner{\ph_\la}{\signal})_\la$
being the analysis operator to some  frame  $(\ph_\la)_{\la \in \La}$.
The functional $\nlf(\signal) =  \sum_{\la \in \La} v_\la \abs{\signal_\la}^q$ is a
weighted  $\ell^q$-norm. The frame $(\ph_\la)_{\la \in \La}$ may be a  prescribed wavelet or curvelet
basis \cite{candes2004,daubechies1988orthogonal,mallat2009wavelet}
or a  trained dictionary \cite{aharon2006ksvd,gribonval2010dictionary}.
In Section~\ref{sec:ellq}, we analyze a {\cm non-linear} version of  $\ell^q$-regularization,
where  $\inner{\phi_\la}{\edot}$ are replaced by non-linear functionals.
{\cm In this case the resulting regularizer will in general be non-convex even if $q \geq 1$.}

\subsection{Well-posedness and weak convergence}
\label{sec:well}

For the convergence analysis of NETT regularization, we use the following
assumptions on the regularizer  and the
data consistency term in \eqref{eq:nett}.

\begin{condition}[Convergence of NETT regularization]\label{cond:main} \mbox{}
\begin{enumerate}[leftmargin=3em,label = (A\arabic*)]
 \item \label{cond:main1} \textbf{Network regularizer $\reg$:}
 \begin{itemize}[wide]
\item
The regularizer  $\reg(\Wo, \edot)$ is defined by
 \eqref{eq:cnn} and \eqref{eq:cnn2}

\item $\Wo_\ell\colon \XXX_{\ell-1} \to \XXX_{\ell-1/2}$ are affine operators of the form
$\Wo_\ell (\signal)  =  \Ao_\ell \signal + \bias_\ell$;

\item
$\Ao_\ell$ are bounded linear;
\item
$\nlo_\ell$ are weakly continuous;
\item
The functional $\nlf$ is weakly lower semi-continuous.
\end{itemize}

\item \label{cond:main2}
 \textbf{Data  consistency term $\simm$:}
 \begin{itemize}[wide]
\item
For some $\tau \geq 1$ we have $\forall  \data_0,\, \data_1,\, \data_2 \in \YY
\colon
 \simm(\data_0, \data_1) \le \tau\simm(\data_0, \data_2) + \tau\simm(\data_2, \data_1)$;

\item
$\forall  \data_0,\, \data_1  \in \YY \colon \simm(\data_0, \data_1) = 0 \iff \data_0 = \data_1$;

\item
$\forall(\data_k)_{k \in \N} \in \YY^{\N} \colon \data_k \to \data \implies \simm(\data_k, \data) \to 0$;
\item
The functional $(\signal, \data) \mapsto \simm(\Fo(\signal), \data)$ is sequentially lower semi-continuous.
\end{itemize}

{\cm
\item \label{cond:coercive1}
 \textbf{Coercivity condition:}\\
$\reg(\Wo, \edot)$ is coercive, that is $\reg(\Wo, \signal) \to \infty$
as $\norm{\signal} \to \infty$.
}
\end{enumerate}
\end{condition}

{ \cm
The   conditions in  \ref{cond:main1}  guarantees the  lower semicontinuity of the regularizer.
The conditions in \ref{cond:main2} for the data consistency term are  not very restrictive  and, for example, are satisfied for the squared norm distance.
The coercivity condition  \ref{cond:coercive1}    might be the  most  restrictive condition. Several strategies how it can be obtained  are
discussed in the following. }

{ \cm
\begin{remark}[Coercivity via skip or residual connections]  \label{rem:c-skip}
Coercivity  \ref{cond:coercive1}  clearly holds for regularizers of the form
\begin{equation} \label{eq:c-skip}
	 \reg(\Wo, \edot) =  \reg^{(1)}(\Wo, \edot) +   \nlf^{(2)}(\signal)  \,,
\end{equation}
where $\reg^{(1)}(\Wo, \edot)$  is a trained regularizer as in \ref{cond:main1}
and $\nlf^{(2)}$  is  coercive and weakly lower semi-continuous.
The regularizer \eqref{eq:c-skip} fits to our general framework and results from
a network using a skip connection between the input and the output layer. In this case, the overall network takes the form $\NN(\Wo,\signal) = [\NN^{(1)}(\Wo,\signal), \signal]$ where  $\NN^{(1)}(\Wo,\signal)$ is of the form \eqref{eq:cnn2}.

Another possibility  to obtain coercivity is to use a residual connection
in the network structure which results  in a regularizer of the  form
\begin{equation} \label{eq:c-res}
    \reg(\Wo, \signal) = \nlf \bigl( \NN^{(r)}(\Wo,\signal) - \signal  \bigr) \,.
 \end{equation}
 If the last non-linearity $\nlo_\ell$  in the network
 $\NN^{(r)}(\Wo,\signal)$ is a bounded function and the functional
 $\nlf$ is coercive, then the resulting regularizer is coercive.
  Coercivity also holds if $\NN^{(r)}(\Wo,\signal)$ has Lipschitz constant $<1$, which
  can be achieved by appropriate training~\cite{behrmann2019invertible}.
\end{remark}
}

{ \cm
\begin{remark}[Layer-wise coercivity]  \label{rem:c-each}
A set of specific conditions  that implies coercivity of the regularizer  is
to assume that, for all $\ell$, the activation functions $\nlo_\ell$
are coercive and there exists $c_\ell \in [0, \infty)$ such that
$\forall \signal \in \XX \colon \norm{\signal} \le c_\ell \norm{\Ao_\ell \signal}$.
The coercivity of $\Ao_\ell $ can be obtained by including a skip  connection,
in which case the operator  $\Ao_\ell$ takes the form $\Ao_\ell (\signal) = [\Ao_{\ell}^{(1)} (\signal), \signal] $,
where $\Ao_\ell$ is bounded linear.

In CNNs, the spaces  $\XXX_\ell$  and $\XXX_{\ell-1/2}$  are function spaces  (see Remark~\ref{rem:cnn})
and a standard operation for $\nlo_\ell$ is the ReLU (the rectified linear unit),  $\operatorname{ReLU} (x) \coloneqq
\max \set{x,0}$, that is applied component-wise.
The  plain form of the ReLU   is not coercive. However, the slight modification $x \mapsto \max\set{x, ax}$
for  some $a \in (0,1)$, named leaky ReLU, is coercive,  see~\cite{maas2013rectifier,he2015delving}.
Another coercive standard operation for $\nlo_\ell$ in CNNs  is max pooling which takes the maximum value $\max\set{ \abs{x(i)} \colon i \in I_k}$  within clusters of transform coefficients.  We emphasize however that by using  one of
the strategies described  in Remark \ref{rem:c-skip}, one can use any common activation function without worrying about its coercivity.
\end{remark}
}

{\cm
\begin{remark}[Generalization of the coercivity condition]\label{rem:c4}
 The results derived  below also hold under the following weaker alternative
 to the coercivity condition  \ref{cond:coercive1}  in Condition \ref{cond:main}:
\begin{enumerate}[leftmargin=3em,label = (A\arabic*')] \setcounter{enumi}{2}
\item \label{cond:coercive2}
For all $ \data \in \YY$ and  $\alpha > 0$, there exists a  $C>0$ such that
\begin{equation}\label{eq:lset}
\set{\signal \in \XX \mid \simm(\Fo (\signal), \data)
 + \alpha \reg(\Wo,\signal) \le C} \quad \text{ is nonempty and bounded  in  $\XX$.}
\end{equation}
\end{enumerate}
Condition \ref{cond:coercive2} ensures that the level set in \eqref{eq:lset} is sequentially weakly pre-compact  for all $\data \in \YY$ and $\alpha > 0$. It is indeed weaker than Condition \ref{cond:coercive1}. For instance, in case that $\Fo$ is linear and $\simm(\edot, 0)$ is convex, \ref{cond:coercive2}  amounts to require that $\reg(\Wo, \edot)$ is coercive on the null space of $\Fo$, whereas  \ref{cond:coercive1}  requires coercivity of $\reg(\Wo, \edot)$   on the whole space  $\XX$.
\end{remark}
}

\begin{theorem}[Well-posedness of CNN-regularization]\label{thm:well}
Let Condition~\ref{cond:main} be satisfied.
Then the following assertions hold true:
\begin{enumerate}
\item
\emph{Existence:} For all $\data \in \YY$ and $\alpha >0$, there
exists a minimizer of $\tik_{\alpha;\data}$;
\item
\emph{Stability:} If $\data_k \to \data$ and $\signal_k \in \argmin  \tik_{\alpha; \data_k}$,
then weak accumulation points of $(\signal_k)_{k \in \N}$ exist and are minimizers of $\tik_{\alpha;\data}$.
\item\label{th:Tikhc}
\emph{Convergence:}  Let $\signal  \in \XX$, $\data \coloneqq \Fo(\signal)$,
 $(\data_k)_{k \in \N}$ satisfy $\simm(\data_k,\data), \simm(\data, \data_k) \le \delta_k$ for some sequence $(\delta_k)_{k\in \N} \in (0, \infty)^\N$ with
  $\delta_k \to 0$,  suppose $\signal_k \in \argmin_\signal\tik_{\alpha(\delta_k)}(\signal, \data_k)$, and let the  parameter choice $\alpha \colon (0, \infty) \to (0, \infty) $ satisfy
\begin{equation} \label{eq:alpha}
\lim_{\delta \to 0} \alpha(\delta) = \lim_{\delta \to 0} \frac{\delta}{\alpha(\delta)} = 0 \,.
\end{equation}
Then the following holds:
\begin{itemize}[wide]
\item
Weak accumulation points of $(\signal_k)_{k \in \N}$ are $\reg(\Wo, \edot)$-minimizing solutions of
$\Fo(\signal) = \data$;
\item
$(\signal_k)_{k \in \N}$ has at least one weak  accumulation point $\signalP$;
\item
Any weakly convergent subsequence $(\signal_{k(n)})_{n \in \N}$
satisfies  $\reg(\Wo, \signal_{k(n)}) \to \reg(\Wo,\signalP)$;
\item
If the $\reg(\Wo, \edot)$-minimizing solution of~$\Fo(\signal) = \data$ is unique, then $\signal_k \rightharpoonup \signalP$.
\end{itemize}
\end{enumerate}
\end{theorem}

\begin{proof}
According to \cite{Gra10,scherzer2009variational} it is sufficient to show that
the  functional $\reg(\Wo, \edot)$ is weakly sequentially lower semi-continuous and the
set $\set{\signal \mid \tik_{\alpha;\data}(\signal) \le t}$ is sequentially weakly
pre-compact  for all $t>0$ and $\data \in \YY$ and $\alpha > 0$.
By the Banach-Alaoglu theorem, the latter condition is satisfied if $\reg(\Wo, \edot)$ is coercive. The coercivity  of $\reg(\Wo, \edot)$  however is assumed
Condition~\ref{cond:main} (for sufficient coercivity conditions see Remarks~\ref{rem:c-skip} and \ref{rem:c-each}). Also from
Condition~\ref{cond:main} it follows that
  $\reg(\Wo, \edot)$ is sequentially lower semi-continuous.
\end{proof}

{\cm Note that the   convergence and stability results of  Theorem~\ref{thm:well}
are valid for any test data independent  of the training data used for optimizing the network regularizer. Clearly, if the considered inverse problem is positive weights, a
$\reg(\Wo, \edot)$-minimizing solutions is not necessarily the one corresponding to the desired signal class for test data very different from this class.}

\subsection{Absolute Bregman distance and total non-linearity}

For convex regularizers, the notion of Bregman
distance  is a powerful concept \cite{burger2004convergence,scherzer2009variational}. For non-convex  regularizers, the standard definition of the Bregman distance
takes negative values. In this paper, we therefore use the  notion of  absolute
Bregman distance. To the best of our knowledge, the absolute
Bregman distance  has not been used in regularization theory  so far.

\begin{definition}[Absolute Bregman distance] \label{def:bregman}
Let $\fun  \colon  \DD \subseteq \XX \to \R$
be G\^ateaux differentiable at $\signal \in \XX$.
The \emph{absolute Bregman distance}
$\breg_{\fun}(\edot, \signal) \colon  \DD  \to [0, \infty]$
with respect to $\fun$ at  $\signal$ is defined by
\begin{equation}\label{eq:abreg} 
\forall \tilde\signal \in \XX \colon \quad
\breg_{\fun}(\tilde\signal, \signal) \coloneqq \abs{\fun(\tilde\signal) - \fun(\signal) - \fun'(\signal)(\tilde\signal - \signal)}.
\end{equation}
Here $\fun'(\signal)$ denotes the G\^ateaux derivative of $\fun$ at $\signal$.
\end{definition}

From Theorem~\ref{thm:well}  we can conclude convergence of $\signal_{\alpha,\delta}$
to the exact solution in the absolute Bregman distance. Below we show that this
implies strong convergence under some additional assumption on the regularization
functional. For this   purpose we introduce the new total non-linearity,
which has not been studied before.

\begin{definition}[Total non-linearity]\label{def:total}
Let $\fun \colon  \DD \subseteq \XX \to \R$
be G\^ateaux differentiable at $\signal \in  \DD$.  We define the \emph{modulus of total non-linearity} of $\fun$ at  $\signal$ as $\modt_{\fun}(\signal, \cdot): [0,\infty) \to [0, \infty]$,
\begin{equation}\label{eq:modulus}
	\forall t > 0 \colon \quad
	\modt_{\fun}(x, t) \coloneqq \inf\set{\breg_{\fun} (\tilde\signal, \signal)\mid \tilde\signal \in \DD \wedge \norm{\tilde\signal - \signal} = t}
	\,.
\end{equation}
The function $\fun$ is called \emph{totally non-linear} at $\signal$ if $\modt_{\fun}(\signal, t) > 0$ for all $t \in (0, \infty).$
\end{definition}

The notion of total non-linearity is similar to total convexity~\cite{BuIu00} for convex functionals.
Opposed to total convexity we do not assume convexity
of $\fun$, and use the absolute Bregman distance instead of the standard Bregman distance.
For convex functions,  the total non-linearity
 reduces to total convexity, as the Bregman distance
is always non-negative for convex functionals.
For a G\^ateaux differentiable convex function, the total non-linearity essentially requires that its second derivative at $\signal$ is bounded away from zero. The functional $\fun(\signal) \coloneqq \sum_{\lambda \in \Lambda} v_\lambda \abs{\signal_\lambda}^q$ with $v_\lambda > 0$ is totally non-linear at every $\signal = (\signal_\lambda)_{\lambda \in \Lambda} \in \ell^\infty(\Lambda)$ if $q > 1$.

We have the following result, which generalizes~\cite[Proposition 2.2]{Res04}
(see also \cite[Theorem 3.49]{scherzer2009variational}) from the convex to
the non-convex case.

\begin{proposition}[Characterization of total non-linearity]\label{prop:tn}
For  $\fun \colon  \DD \subseteq \XX \to \R$
and any $\signal \in \DD$ the following
assertions are equivalent:
\begin{enumerate}[label = (\roman*)]
\item\label{prop:tn1}
The function $\fun$ is totally non-linear at $\signal$;
\item\label{prop:tn2}
$\forall \kl{\signal_n}_{n \in \N}\subseteq \DD^{\N}
\colon
(\lim_{n \to \infty} \breg_{\fun}(\signal_n, \signal) = 0  \wedge (\signal_n)_n \text{ bounded} ) \Rightarrow  \lim_{n \to \infty} \norm{\signal_n - \signal} = 0$.
\end{enumerate}
\end{proposition}

\begin{proof}
The proof of the implication \ref{prop:tn2} $\Rightarrow$ \ref{prop:tn1} is the same as~\cite[Proposition 2.2]{Res04}. For the implication \ref{prop:tn1} $\Rightarrow$  \ref{prop:tn2}, let \ref{prop:tn1} hold, let
$\kl{\signal_n}_{n \in \N}\subseteq \DD^{\N}$ satisfy $\breg_{\fun}(\signal_n, \signal) \to  0$,  and suppose $\lim_{n\to\infty} \norm{\signal_n-\signal} = \delta > 0$ for the moment. For any $\varepsilon>0$, by the continuity of $\breg_{\fun}(\edot, \signal)$, there exist $\tilde\signal_n$ with $\norm{\tilde\signal_n - \signal} = \delta$ such that for sufficiently large $n$
\[
\varepsilon \ge \breg_{\fun}(\signal_n, \signal) + \frac{\varepsilon}{2} \ge \breg_{\fun}(\tilde\signal_n, \signal) \ge \modt_{\fun}(\signal, \delta).
\]
This leads to $\modt_{\fun}(\signal, \delta) = 0$, which contradicts with the total non-linearity
of $\fun$ at $x$. Then, the assertion follows by considering subsequences of $\kl{\signal_n}_{n\in\N}$.
\end{proof}

\begin{remark}
We point out that Proposition~\ref{prop:tn} remains true, if we replace the absolute value $\abs{\edot} \colon  \R \to [0, \infty]$ in~\eqref{eq:abreg} by the
ReLU function, the leaky ReLU function,
or any other nonnegative continuous function $\phi \colon  \R \to [0, \infty]$ that satisfies
$\phi(0) = 0$.
\end{remark}

\subsection{Strong convergence of NETT regularization}
\label{sec:string}

For totally non-linear regularizers $\reg(\Wo, \edot)$
we can prove  convergence  of NETT with respect to the  norm topology.

\begin{theorem}[Strong convergence of NETT]\label{thm:strong}
Let Condition~\ref{cond:main} hold and assume additionally
that $\Fo(\signal) = \data$ has a solution,
$\reg(\Wo, \edot)$ is totally non-linear at $\reg(\Wo, \edot)$-minimizing solutions, and $\alpha$ satisfies \eqref{eq:alpha}.
Then for every sequence  $(\data_k)_{k \in \N}$ with
 $\simm(\data_k,\data), \simm(\data, \data_k) \leq \delta_k$ where $\delta_k \to 0$  and every sequence $\signal_k \in \argmin_\signal\tik_{\alpha(\delta_k)}(\signal, \data_k)$,
 there exist a subsequence $(\signal_{k(n)})_{n \in \N}$ and an $\reg(\Wo, \edot)$-minimizing solution $\signalP$
 with $\snorm{\signal_{k(n)} - \signalP} \to 0$.
If the $\reg(\Wo, \edot)$-minimizing solution is unique, then $\signal_k \to \signalP$
with respect to the norm topology.
\end{theorem}

\begin{proof}
It follows from Theorem~\ref{thm:well} that
there exists a subsequence $(\signal_{k(n)})_{n \in \N}$ weakly converging to some
$\reg(\Wo, \edot)$-minimizing solution
$\signalP$ such that $\reg(\Wo, \signal_{k(n)}) \to \reg(\Wo,\signalP)$.
From the weak convergence of $(\signal_{k(n)})_{n \in \N}$ and the convergence of
$(\reg(\Wo, \signal_{k(n)}))_{n \in \N}$ it follows that $\breg_{\reg(\Wo, \edot)}(\signal_{k(n)},\signalP) \to 0$.
Thus it follows from Proposition~\ref{prop:tn}, that $\snorm{\signal_{k(n)}-\signalP} \to 0$.
If $\signalP$ is the unique $\reg$-minimizing solution,
the strong convergence to $\signalP$ again follows from Theorem~\ref{thm:well} and Proposition~\ref{prop:tn}.
\end{proof}

\section{Convergence rates}
\label{sec:rates}

In this section, we derive convergence rates for NETT
in terms of general error measures under certain variational inequalities.
{\cm Throughout we denote by $\delta > 0$ the noise level and $\alpha > 0$ the regularization parameter.}
We discuss instances where the variational inequality is satisfied for
the absolute Bregman distance.
Additionally, we consider a {\cm non-linear} generalization
of $\ell^q$-regularization.

\subsection{General convergence rates result}

We study   convergence rates  in terms of a general
functional $\err \colon \XX \times \XX \to [0, \infty]$ measuring
closeness in the space $\XX$. For convex $\Psi \colon [0, \infty) \to [0, \infty)$, let
$\Psi^\ast \colon \R \to \R$ denote the Fenchel conjugate of $\Psi$
defined by $\Psi^*(\tau)  \coloneqq \sup \set{ \tau \, t - \Psi(t) \mid t  \geq 0 }$.

\begin{theorem}[Convergence rates for NETT]\label{thm:rate}
Suppose $\err \colon \XX \times \XX \to [0, \infty]$, let $\signalP \in \DD$ and assume  that
there exist    a concave, continuous and strictly increasing function $\Phi \colon [0, \infty) \to [0, \infty)$
with $\Phi(0) = 0$ and a constant $\beta >0$  such that
for all $\varepsilon > 0$ and  $\signal \in  \DD$ with $\abs{\reg(\Wo, \signal) - \reg(\Wo, \signalP)} < \varepsilon$ we have
\begin{equation}\label{eq:vin}
\beta \err(\signal, \signalP) \le \reg(\Wo, \signal) - \reg(\Wo, \signalP) + \Phi(\simm(\Fo(\signal), \Fo(\signalP))) \,.
\end{equation}
Additionally, let Condition~\ref{cond:main} hold, let $\data_\delta\in\YY$ and $\delta>0$ satisfy
 $\simm(\data, \data_\delta), \simm(\data_\delta, \data) \le \delta$
 and write  $\Phi^{-*}$ for the Fenchel conjugate of the inverse function $\Phi^{-1}$.
Then the following assertions hold true:

\begin{enumerate}
\item\label{thm:rate-a}  For sufficiently small $\alpha$ and $\delta$, we have
\begin{equation}\label{eq:rate1}
\forall  \signal_{\alpha,\delta} \in \argmin \tik_{\alpha; \data_\delta}  \colon
\quad
\beta \err(\signal_{\alpha,\delta}, \signalP) \le \frac{\delta}{\alpha} + \Phi(\tau\delta) + \frac{\Phi^{-*}(\tau\alpha)}{\tau\alpha} \,.
\end{equation}

\item\label{thm:rate-b} If $\alpha \sim \delta / \Phi(\tau\delta)$, then
$\err(\signal_{\alpha,\delta}, \signalP) = \bgO\left(\Phi(\tau\delta)\right)$
as $\delta \to 0$.
\end{enumerate}
\end{theorem}

\begin{proof}  \mbox{}
\ref{thm:rate-a} By Theorem~\ref{thm:well}~\ref{th:Tikhc}, we can  assume that $\sabs{\reg(\Wo, \edot)(\signal_{\alpha,\delta}) - \reg(\Wo, \signalP)} \le \varepsilon$.
From the  definition of $\signal_{\alpha,\delta}$ ite follows that $\simm(\Fo (\signal_{\alpha,\delta}), \data_\delta) + \alpha \reg(\Wo, \edot)(\signal_{\alpha,\delta}) \le \simm(\Fo (\signalP), \data_\delta) + \alpha \reg(\Wo, \signalP)$. Then from~\eqref{eq:vin} we obtain
\begin{align*}
\alpha\beta\err(\signal_{\alpha,\delta}, \signalP) & \le \delta - \simm(\Fo(\signal_{\alpha,\delta}), \data_\delta) + \alpha\Phi(\simm(\Fo(\signal_{\alpha,\delta}), \Fo(\signalP)))\\
& \le \delta - \simm(\Fo(\signal_{\alpha,\delta}), \data_\delta) + \alpha\Phi(\tau\delta+\tau\simm(\Fo(\signal_{\alpha,\delta}), \data_\delta)) \\
& \le \delta - \simm(\Fo(\signal_{\alpha,\delta}), \data_\delta) + \alpha\Phi(\tau\delta)+\alpha\Phi(\tau\simm(\Fo(\signal_{\alpha,\delta}), \data_\delta))\\
& \le \delta + \alpha\Phi(\tau\delta) + \tau^{-1}\Phi^{-*}(\tau\alpha).
\end{align*}
The last estimate is an application of Young's inequality
$\alpha\Phi(\tau t) \le t + \tau^{-1}\Phi^{-*}(\tau \alpha)$.

\ref{thm:rate-b}
Elementary computations show $\limsup_{\delta \to 0} \Phi^{-*}\bigl(\tau \delta /\Phi(\tau\delta)\bigr)/\delta < \infty$, such that
the right hand side of~\eqref{eq:rate1} is bounded by $\Phi(\tau\delta)$ up to a constant if $\alpha \sim \delta / \Phi(\tau\delta)$.
\end{proof}

\begin{remark}
If $\Phi(t) \le C t$ for sufficiently small $t$, then from~\eqref{eq:rate1} it follows that $\beta \err(\signal_{\alpha,\delta},\signal) \le \delta/\alpha + C \tau \delta = \bgO(\delta)$ if $\alpha \le 1/C$.
This says that the regularization parameter $\alpha$ needs not to vanish for $\delta \to 0$, which is often referred to
as  exact penalization (for the convex case, see the discussions in ~\cite{scherzer2009variational,burger2004convergence}).
\end{remark}

\subsection{Rates in the absolute Bregman distance}

We next derive  conditions under which a variational inequality in form of~\eqref{eq:vin} is possible for the absolute Bregman distance as error measure, $\err(\signal,\signalP) \coloneqq \breg_{\reg(\Wo, \edot)}(\signal, \signalP)$.

\begin{proposition}[Rates in the absolute Bregman distance] \label{prop:cs}
Let $\XX$ and $\YY$ be Hilbert spaces and let $\Fo \colon \XX \to \YY$ be a bounded linear operator. Assume that $\reg(\Wo, \edot)$ is G\^ateaux differentiable, that $\reg'(\Wo, \signalP) \in \ran(\Fo^*)$, and that there exist positive constants $\gamma$, $\varepsilon$ with
\begin{equation}\label{eq:tc}
\reg(\Wo, \signalP) - \reg(\Wo, \signal)\le \gamma\snorm{\Fo(\signal) - \Fo(\signalP)}
\end{equation}
for all $\signal$ satisfying $\sabs{\reg(\Wo, \signal) - \reg(\Wo, \signalP)} < \varepsilon$.
Then,
\[
\breg_{\reg(\Wo, \edot)}(\signal, \signalP) \le \reg(\Wo,\signal) - \reg(\Wo,\signalP) + C\snorm{\Fo (\signal) - \Fo(\signalP)}\qquad\text{for some constant }C.
\]
In particular, for the similarity measure   $\simm(\zsignal, \data) = \snorm{\zsignal - \data}^2$ and under  Condition~\ref{cond:main},
 Items \ref{thm:rate-a} and \ref{thm:rate-b} of Theorem~\ref{thm:rate} hold true.
\end{proposition}

\begin{proof}
Let $\xi$ satisfy that $\reg'(\Wo, \signalP) = \Fo^*\xi$. Then $\sabs{\sinner{\reg'(\Wo, \signalP)}{\signal - \signalP}} \le  \snorm{\xi}\snorm{\Fo(\signal) - \Fo(\signalP)}$. Note that
$\sabs{\reg(\Wo, \signal) - \reg(\Wo, \signalP)} = \reg(\Wo, \signal) - \reg(\Wo, \signalP) $ if
$ \reg(\Wo, \signal) \ge \reg(\Wo, \signalP)$ and
$
 \lvert \reg(\Wo, \signal) - \reg(\Wo, \signalP) \rvert=
\reg(\Wo, \signal) - \reg(\Wo, \signalP) + 2(\reg(\Wo, \signalP) - \reg(\Wo, \signal)) \le \reg(\Wo, \signal) - \reg(\Wo, \signalP) + 2\gamma\snorm{\Fo(\signal) - \Fo(\signalP)}
$ otherwise. This yields
\begin{multline*}
\breg_{\reg(\Wo, \edot)}(\signal, \signalP) \le \sabs{\reg(\Wo, \signal) - \reg(\Wo, \signalP)} +\sabs{\sinner{\reg'(\Wo, \signalP)}{\signal - \signalP}} \\
\leq \reg(\Wo, \signal) - \reg(\Wo, \signalP) + C\snorm{\Fo (\signal) - \Fo (\signalP)}
\end{multline*}
with the  constant $C \coloneqq \norm{\xi}+2\gamma$, and concludes the proof.
\end{proof}

\begin{remark}
Proposition~\ref{prop:cs} shows that a variational inequality of the form~\eqref{eq:vin} with $\beta = 1$ and $\Phi(t) = C\sqrt{t}$ follows from a classical source condition $\reg'(\signalP) \in \ran(\Fo^*)$. By Theorem~\ref{thm:rate}, it further implies that $\breg_{\reg(\Wo, \edot)}(\signal_{\alpha,\delta}, \signalP) = \mathcal{O}({\delta})$ if $\snorm{\data - \data_\delta} \le \delta$. Moreover, we point out that the additional assumption~\eqref{eq:tc} is rather weak, and follows from the classical source condition $\reg'(\signalP) \in \ran(\Fo^*)$  if $\reg$ is convex, see~\cite{grasmair2008sparse}. It is clear that a sufficient condition to~\eqref{eq:tc} is
\[
\sabs{\reg(\Wo, \signal) - \reg(\Wo, \signalP) - \sinner{\reg'(\Wo, \signalP)}{\signal - \signalP}} \le c \sabs{\reg(\Wo, \signal) - \reg(\Wo, \signalP)}\qquad \text{ for some } c < 1,
\]
which resembles a tangential-cone condition.
 {\cm
 Choosing the  squared Hilbert space norm  for the similarity measure,
 $\simm(\zsignal, \data) = \snorm{\zsignal - \data}^2$, the error estimate  \eqref{eq:rate1} takes  the  form
\begin{equation}\label{eq:rate1_bregman}
    \forall  \signal_{\alpha,\delta} \in \argmin \tik_{\alpha; \data_\delta}  \colon
    \quad
     \breg_{\reg(\Wo,\edot)}(\signal_{\alpha,\delta}, \signalP) \le \frac{\delta}{\alpha} + C\sqrt{\delta} + \frac{C^4}{4}\alpha \,.
\end{equation}
In particular,  choosing $\alpha \sim \sqrt{\delta}$ yields the convergence  rate
$\breg_{\reg(\Wo,\edot)}(\signal_{\alpha,\delta}, \signalP) =
\mathcal{O}(\sqrt{\delta})$.}
\end{remark}

\subsection{General regularizers}
\label{sec:generalizations}

So far we derived  well-posedness, convergence and convergence rates
for  regularizers  of the form \eqref{eq:cnn}.  These results
can be generalized to Tikhonov  regularization
 \begin{equation}\label{eq:tik}
\tik_{\alpha;\data_\delta}(\signal)\coloneqq
\simm (\Fo (\signal), \data_\delta) + \alpha \reg (\signal) \to \min_\signal \,,
\end{equation}
where the regularization term is not necessarily defined by a neural network.
These  results are derived by replacing
Condition~\ref{cond:main} with the following one.

\begin{condition}[Convergence for general regularizers]
\label{cond:gen} \mbox{}
\begin{enumerate}[leftmargin=3em,label = (B\arabic*)]
\item \label{cond:gen1}
The functional $\reg$ is sequentially lower semi-continuous.
\item \label{cond:gen2}
The set $\set{\signal \mid \tik_{\alpha;\data}(\signal) \le t}$ is sequentially pre-compact for all $t, \data$ and $\alpha > 0$.

\item \label{cond:gen3} The data consistency term satisfies \ref{cond:main2}.
\end{enumerate}
\end{condition}

Then we have the following:

\begin{theorem}[Results for general Tikhonov regularization] \label{thm:gen}
Under Condition~\ref{cond:gen},  general Tikhonov
regularization \eqref{eq:tik} satisfies the following:
\begin{enumerate}
\item\label{thm:gen-1} The conclusions from  Theorem \ref{thm:well}
(well-posedness and convergence) hold true.
\item\label{thm:gen-2} If $\reg$ is totally non-linear at $\reg$-minimizing  solutions, the strong  convergence from
Theorem~\ref{thm:strong} holds.
\item\label{thm:gen-3} The convergence rates result from  Theorem \ref{thm:rate} holds.

\item\label{thm:gen-4} If $\Fo$ is bounded linear,
the assertions of  Proposition~\ref{prop:cs} hold for $\reg$.

\end{enumerate}
\end{theorem}

\begin{proof}
All assertions are shown as in the special case $\reg = \reg(\Wo, \edot)$.
\end{proof}

Note that Item~\ref{thm:gen-1} in the above theorem is contained
in \cite{Gra10}. Items~\ref{thm:gen-2}-\ref{thm:gen-4} have not been obtained  previously for non-convex regularizers.

\subsection{Non-linear $\ell^q$-regularization}
\label{sec:ellq}

We now analyze a special instance of  NETT regularization~\eqref{eq:cnn},
generalizing classical $\ell^q$-regularization {\cm by including non-linear transformations.}
More precisely, we consider the
following $\ell^q$-Tikhonov functional
\begin{equation} \label{eq:Lqreg}
\tik_{\alpha, \data_\delta} (\signal) = \snorm{\Fo(\signal) - \data_\delta}^2 + \alpha \sum_{\lambda \in \Lambda} v_\lambda \abs{\phi_\lambda(\signal)}^q \qquad \text{ with } q >1.
\end{equation}
Here $\Lambda$ is a countable set and $\phi_\lambda \colon \XX \to \R$ are possibly non-linear functionals.
{\cm Theorem \ref{thm:well} assures existence and convergence  of minimizers of \eqref{eq:Lqreg} provided that $(\phi_\lambda)_{\la \in \La}$ is coercive and weakly continuous. If $(\phi_\lambda)_{\la \in \La}$   is non-linear, minimizers are not
necessarily  unique.}

The  regularizer  $\reg(\signal) \coloneqq \sum_{\lambda \in \Lambda} v_\lambda \abs{\phi_\lambda(\signal)}^q$ is a particular instance of  NETT  \eqref{eq:cnn}, if we take
$\nlo_{L} \circ \Wo_L \circ  \cdots \circ \nlo_1 \circ \Wo_1 = (\phi_\lambda)_{\lambda \in \Lambda}  $, and  $\nlf$  as a weighted $\ell^q$-norm. However, in \eqref{eq:Lqreg} also more general choices for $\phi_\lambda$ are allowed (see Condition \ref{as:ex}).

We assume the following:

\begin{condition}[]\label{as:ex} \mbox{}
\begin{enumerate}[leftmargin=4em,label = (C\arabic*)]
\item \label{as:ex1}
$\Fo \colon \XX \to \YY$ is a bounded linear operator between Hilbert spaces $\XX$ and $\YY$.
\item \label{as:ex2}
$\phi_\lambda \colon \XX \to \R$ is G\^ateaux differentiable for every $\lambda \in \Lambda$.
\item \label{as:ex3}
There is a $\reg(\Wo, \edot)$-minimizing solution $\signalP$ with
$\reg'(\Wo,\signalP) \in \ran(\Fo^*)$.
\item \label{as:ex4}
There exist  constants $C, \varepsilon >0 $ such that for all $\signal$ with $\sabs{\reg(\Wo,\signal) - \reg(\Wo,\signalP)} \le \varepsilon$, it holds that
\[
\forall \lambda \in \Lambda \colon \quad
\sgn(\phi_\lambda(\signalP))(\phi_\lambda(\signalP) - \phi_\lambda(\signal)) \le C \sgn(\phi_\lambda(\signalP))\phi'_\lambda(\signalP)(\signalP - \signal)  \,.
\]
Here $\sgn(t) = 1$ for $t >0$, $\sgn(t) = 0$ for $t = 0$, and
$\sgn(t) =-1$ otherwise.
\end{enumerate}
\end{condition}

\begin{proposition}\label{prop:ex}
Let Condition~\ref{as:ex} hold, suppose that $\data_\delta \in \YY$ is such that $\norm{\data -\data_\delta} \le \delta$, and let
$\signal_{\alpha,\delta} \in \argmin \tik_{\alpha, \data_\delta}$.
If choosing $\alpha \sim \delta$, then $\breg_{\reg}(\signal_{\alpha,\delta}, \signalP) = \bgO(\delta).$
\end{proposition}

\begin{proof}
The convexity of $t \mapsto \sabs{t}^q$ implies that
\begin{align*}
\reg(\signalP) - \reg(\signal) \le \sum_{\lambda \in \Lambda} v_\lambda \sabs{\phi_\lambda(\signalP)}^{q-1}\sgn(\phi_\lambda(\signalP))(\phi_\lambda(\signalP) - \phi_\lambda(\signal)).
\end{align*}
By~\ref{as:ex4}, we obtain $\reg(\signalP) - \reg(\signal) \le C \sinner{\reg'(\signalP)}{\signalP - \signal}$. This together with~\ref{as:ex3} implies~\eqref{eq:tc}. Thus, the assertion follows from Proposition~\ref{thm:gen}.
\end{proof}

\begin{remark}
Consider the case that $\phi_\lambda(\signal) \coloneqq \sinner{\varphi_\lambda}{\signal}$ for an orthonormal basis $(\varphi_\lambda)_{\lambda \in \Lambda}$ of $\XX$. It is known that $\norm{\signal- \signalP}^2 = \bgO\bigl(\breg_{\reg(\Wo, \edot)}(\signal, \signalP)\bigr)$, see e.g.~\cite{scherzer2009variational}.  Then, Proposition~\ref{prop:ex} gives $\norm{\signal_{\alpha,\delta} - \signalP} = \bgO(\delta^{1/2})$, which reproduces the result of~\cite[Theorem 3.54]{scherzer2009variational}. This rate can be improved to $\bgO(\delta^{1/q})$ if we further assume sparsity of $\signalP$ and restricted injectivity of $\Fo$. It can be shown by Theorem~\ref{thm:rate} because in such a situation~\eqref{eq:vin} holds with $\Phi(t) \sim \sqrt{t}$ and $\err(\signal,\signalP) = \snorm{\signal - \signalP}^q$, see~\cite{grasmair2008sparse} for details.
\end{remark}

{\cm
\subsection{Comparison to the $W$-Bregman distance}

A different framework for deriving convergence rates
for non-convex regularization functionals is
based on the $W$-Bregman distance introduced in~\cite{Gra10,GraHalSch11b}.  In this subsection we compare our absolute
Bregman distance with the $W$-Bregman for some specific examples.

\begin{definition}[$W$-Bregman distance]
Let  $W$ be a set of functions $w\colon \XX \to \R $,
$\reg  \colon \XX  \to [0, \infty]$ be a functional and $\signalP \in \XX$.

\begin{enumerate}
\item
$\reg$ is called $W$-convex at $\signalP$ if
$\reg(\signalP) = \sup_{w \in W} \inf_{\signal \in \XX} ( \reg(\signal) - w(\signal) + w(\signalP))$.

\item
Let $\reg$ be  $W$-convex at $\signalP$. The $W$-subdifferential of $\reg$ at $\signalP$ is defined by  $\partial_W \reg(\signalP) \coloneqq \set{w \in W \mid \forall \signal \in \XX \colon
\reg(\signal) \geq \reg(\signalP) + w(\signal) - w(\signalP)}$.
Moreover, the $W$-Bregman distance
with respect to $w \in \partial_W \reg(\signalP)$ between $\signalP$ and $\signal \in \XX$  is defined by
\begin{equation*}
	\breg_{\reg,W}^{w}(\signal, \signalP) \coloneqq \reg(\signal)
	-  \reg(\signalP) - w(\signal) + w(\signalP) \,.
\end{equation*}
\end{enumerate}
\end{definition}

The   notion of $W$-Bregman distance reduces to its
classical counterpart if we take  $W$
as the set of all bounded  linear functionals on $\XX$.
Allowing more general function sets $W$ provides an extension of the
Bregman distance to non-convex functionals that can be used as an alterative
approach to convergence rates. It, however, requires finding a suitable
function set such that $\reg$ is $W$-convex at $\signalP$.

Relations between the absolute Bregman distance and the $W$-Bregman distance
depends on the particular choice of $W$. We illustrate this by an example.

\begin{example}
Consider the functional $\reg(\signal)\coloneqq \bigl(2\operatorname{ReLU}(\signal-\signalP)-1\bigr)\abs{\signal-\signalP}^q$  for $x\in\R$ with $q > 1$. The absolute Bregman distance according to Definition~\ref{def:bregman}
is given by $ \breg_{\reg}(\signal, \signalP) = \abs{\signal-\signalP}^q$.
For the $W$-Bregman distance consider the family $W \coloneqq  \set{ w_{\alpha, \beta} \colon \R \to\R \mid \alpha, \beta \in \R}$ where
$w_{\alpha, \beta}(\signal) \coloneqq \alpha (\signal -\signalP) - \beta\abs{\signal -\signalP}^q $.
Then $\reg$ is locally
convex at $\signalP$ with respect to $W$. Moreover, $w_{\alpha, \beta} \in \partial_W \reg(\signalP)$ if $\alpha = 0$ and $\beta \ge 1$. For $w_{0, \beta} \in \partial_W \reg(\signalP)$, it follows that
\begin{equation*}
\breg_{\reg,W}^{w_{0, \beta}}(\signal, \signalP) =
\begin{cases}
(\beta +1) \abs{\signal-\signalP}^q & \text{ if } \signal \ge \signalP \\
(\beta -1) \abs{\signal-\signalP}^q & \text{ otherwise}.
\end{cases}
\end{equation*}
In case of $q = 2$, the $W$-subdifferential is closely related to the notion of proximal subdifferentiability used in \cite{clarke1997nonsmooth}.

If $\beta > 1$, then $\breg_{\reg,W}^{w_{0, \beta}}(\edot,\signalP)$ and $\breg_{\reg}(\edot,\signalP)$ only differ in terms of the front constants. As a consequence,
the rates with respect to both Bregman distances will be of the same order.
However, if $\beta = 1$, then $\breg_{\reg,W}^{w_{0, 1}}(\edot,\signalP)$ equals $0$ when  $\signal \le \signalP$. In contrast,  $\breg_{\reg}(\edot,\signalP)$ always treats both $\signal \ge \signalP$ and $\signal \le \signalP$ equally.
\end{example}

We conclude that the relation between the absolute Bregman distance and  the
$W$-Bregman distance  depends on the particular situation and the choice of the family $W$.
Using the $W$-Bregman distance for the analysis of NETT and studying relations between the two generalized Bregman  distances are interesting lines of research that we aim to address in future work.
}

\section{\cm A data driven regularizer for NETT}
\label{sec:auto}

In this section we present a framework for constructing  a  trained
neural network regularizer $\reg(\Wo, \edot)$ of the form  \eqref{eq:cnn2}.
Additionally, we develop a strategy for network training and
minimizing the NETT functional.

\begin{psfrags}
\psfrag{r}{regularizer $\nlf(\NN(\Wo,\edot))$ }
\psfrag{e}{encoder $\NNe(\WWe,\edot)$}
\psfrag{d}{decoder $\NNd(\WWd,\edot)$}
\begin{figure}[htb!]
\centering     \includegraphics[width=\textwidth]{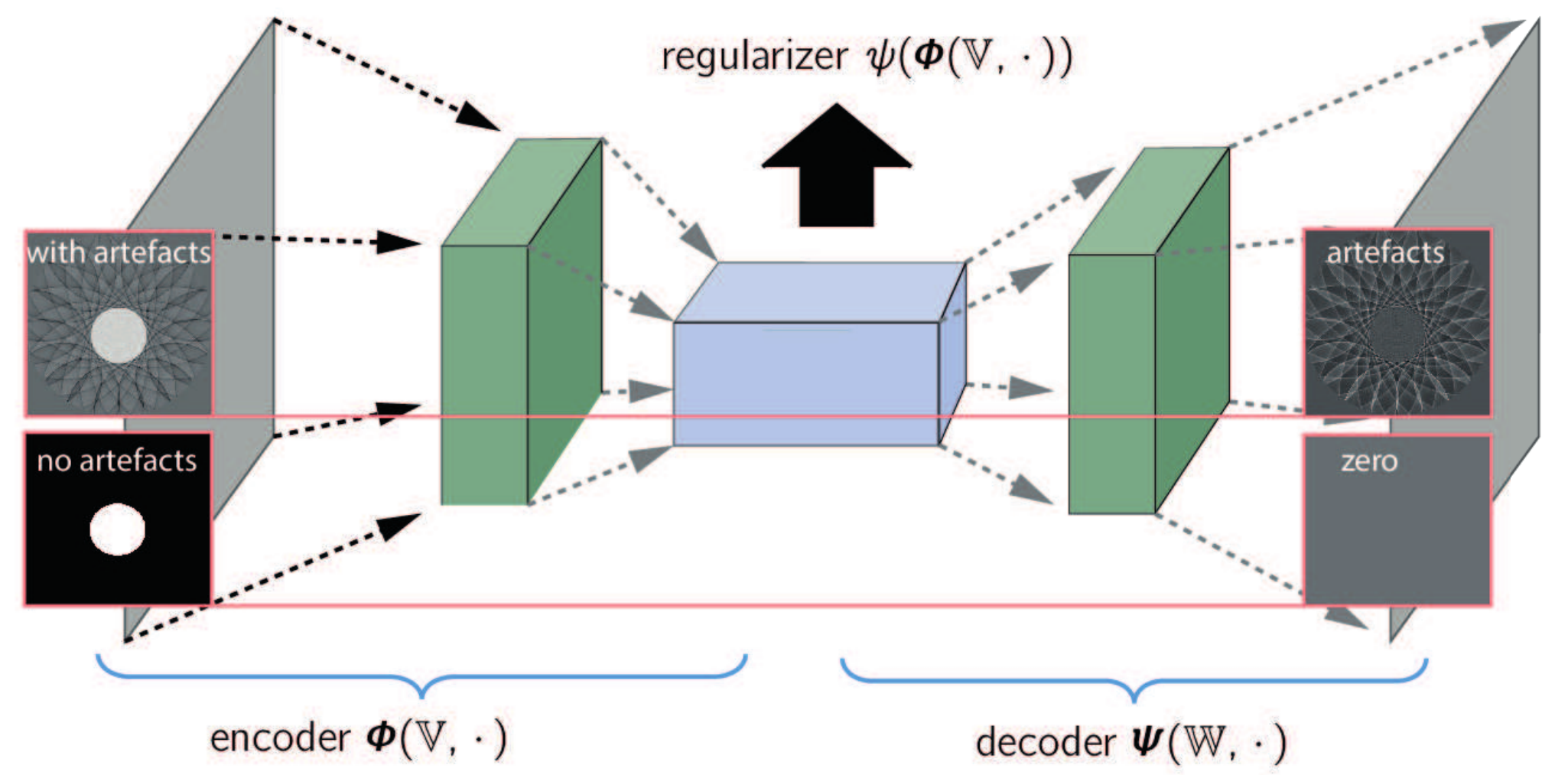}
    \caption{\textbf{Encoder-decoder scheme and proposed training strategy.}\label{fig:aed} The network
    consists of the encoder part $\NNe(\WWe,\edot)$ and decoder part $\NNd(\WWd,\edot)$,
    and is trained to map  any potential solution $\signal$ to the corresponding artifact part. The norm of $\NNe(\WWe,\signal)$
    is used as  trained regularizer.}
    \end{figure}
\end{psfrags}

\subsection{A trained regularizer}
\label{sec:auto2}

For the regularizer we propose $\reg(\Wo,\edot) = \sum_{\lambda\in \Lambda_L} \norm{ \NNe_\lambda(\WWe,\signal)}_q^q $ with a network $\NNe(\WWe,\edot) = (\NNe_\lambda(\WWe,\edot))_{\la \in \La}
$
of the form \eqref{eq:cnn2}, that itself is part  of  encoder-decoder type network
\begin{equation} \label{eq:auto}
\NNd(\WWd, \edot) \circ \NNe(\WWe,\edot) \colon \XX \rightarrow \XX \,.
\end{equation}
Here $\NNe(\WWe,\edot)\colon \XX \to  \XXX_L $ can be interpreted as encoding network  and $\NNd \colon \XXX_L \to  \XX$  as decoding network. {\cm We note, however, that any  network with at least one hidden layer can be written in the form \eqref{eq:auto}.
Moreover we  also allow  $ \XXX_L $ to be of large dimension  in which  case the encoder
$\NNe(\WWe,\edot)$ does not perform any form  of dimensionality reduction or compression.}

Training of the network is performed  such that $\reg(\Wo, \edot)$
is small for artifact free images and large for images
with artifacts. The proposed training strategy is presented
below.

For suitable network training of the encoder-decoder scheme \eqref{eq:auto}, we propose the following strategy (compare Figure \ref{fig:aed}).
We choose a  set of  training  phantoms  $\zzt_\ntrain \in \XX$ for
$\ntrain =1, \dots,  \Ntrain_1 +  \Ntrain_2$ from   which we construct  back-projection images
 $\xxt_\ntrain  \coloneqq  \Fo^\Plus ( \Fo \zzt_\ntrain)$ (where $\Fo^\Plus$ denotes the pseudo-inverse) for the  first $\Ntrain_1$ training examples,
and set $\xxt_\ntrain = \zzt_\ntrain$ for the
last $\Ntrain_1$ training images. From this we define the training data
$\set{(\xxt_\ntrain, \rrt_\ntrain)}_{\ntrain=1}^{\Ntrain_1 + \Ntrain_2}$,
where
\begin{align}\label{eq:z1}
	\rrt_\ntrain &= \zzt_\ntrain - \xxt_\ntrain = \zzt_\ntrain  - \Fo^\Plus (\Fo  \zzt_\ntrain)
	&& \text{ for } \ntrain =1, \dots, \Ntrain_1 \\ \label{eq:z2}
	\rrt_\ntrain &= \zzt_\ntrain - \xxt_\ntrain =0
	&& \text{ for } \ntrain = \Ntrain_1+1, \dots,  \Ntrain_1 + \Ntrain_2  \,.
\end{align}
The free parameters in \eqref{eq:auto} are adjusted in such a way,
that $\NNd(\WWd, \NNe(\WWe, \xxt_\ntrain))  \simeq \rrt_\ntrain$   for any training
pair $(\xxt_\ntrain, \rrt_\ntrain)$.
This is achieved by minimizing the error function
\begin{equation} \label{eq:err}
 	E_N(\WWe,\WWd ) \coloneqq
 	\sum_{n=1}^{\Ntrain_1+\Ntrain_2} \ell(\NNd(\WWd, \NNe(\WWe, \xxt_\ntrain)), \rrt_\ntrain ) \,,
\end{equation}
where $\ell$ is a suitable distance measure  (or loss function) that quantifies the error made  by the network
function on the $\ntrain$-th training sample.
Typical choices for $d$ are mean absolute error or mean squared
error.

Given an arbitrary unknown $\signal \in \XX$, the trained network  estimates the artifact part.  As a consequence,  $\reg(\WWe, \xxt)$ is expected to be
large, if $\xxt$ contains severe artifacts and small if it is almost artifact free. If $\xxt$ is similar to elements in the
training set, this  should produce almost artifact free results with NETT regularization. Even if  the true unknown  is of different type from the training data, artifacts as well as noise will have large value of the regularizer. Thus our approach is applicable for a wider range of images apart from training ones. This claim is confirmed by our numerical results  in Section~\ref{sec:num}.
Note that  we did not explicitly account for the coercivity condition  in \ref{cond:main1}
 during the training phase. {\cm Several possibilities for ensuring coercivity are discussed in Section~\ref{sec:well}. Moreover, note that the class of methods we have in mind for the above training strategy are underdetermined  problems  such as undersampled CT, MRI or PAT. We expect that similar training strategies can be designed for
 problems that have many small but not vanishing singular values.
 Investigating such issues in more detail (theoretically and
numerically) is an interesting line of future research.}

{\cm
\begin{remark}[Alternative trained regularizers]
Another natural choice would be to simply take
$\reg(\WWd,\WWe,\signal) \coloneqq \norm{\NNd(\WWd, \NNe(\WWe,\signal))}^2$
for the regularizer. Such a regularizer has been used in the proceedings
\cite{antholzer2019nett} combined with quite simple network architectures for $\NNd$ and
$ \NNe$. The main emphasis of this paper is the  convergence  analysis of NETT, so the investigation
of effects of different trained regularizers is beyond its scope.
We nevertheless point out that including training data corresponding to \eqref{eq:z2}
makes  the trained network $\NNd(\WWd, \NNe(\WWe,\signal))$ different to standard artifact
removal network \cite{jin2017deep}. Such methods only use training data  corresponding to
\eqref{eq:z2} to  remove artifacts. Detailed investigation of benefits of each approach is an interesting
aspect of future research.
\end{remark}
}

\subsection{Minimizing the NETT functional}
\label{sec:grad}

Using the encoder-decoder scheme, regularized solutions
are defined as minimizers of the NETT functional
\begin{equation} \label{eq:tikU}
\tik_{\alpha;\data} (\signal) =
\frac{1}{2}  \norm{ \Fo (\signal)  - \data_{\delta}}^2 + \alpha
\sum_{\lambda\in \Lambda_L} \norm{ \NNe_\lambda(\WWe,\signal)}_q^q \,,
\end{equation}
where $\NNe$ is trained as  above. The optimization
problem \eqref{eq:tikU} is non-convex (due to the presence of the non-linear network) and non-smooth if $q=1$. Note that  the  subgradient
of the regularization term $\reg(\Wo,\signal) = \sum_{\lambda\in \Lambda_L} \norm{ \NNe(\WWe,\signal)_\lambda}_q^q$
can be evaluated  by standard software for network training with the backpropagation algorithm.
We therefore propose  to use an incremental gradient method  for minimizing  the
Tikhonov functional \eqref{eq:tikU}, which alternates  between a gradient descent  step for
$\frac{1}{2}  \norm{ \Fo (\signal)  - \data_{\delta}}^2$ and a
subgradient descent  step  for the regularizer $\reg(\Wo,\signal) $.

The resulting minimization procedure is summarized
in Algorithm~\ref{alg:gradient}.
\begin{algorithm}
\caption{Incremental\label{alg:gradient} gradient descent  for minimizing NETT}
\begin{algorithmic}[H]
\STATE {Choose family of step-sizes $(s_i)>0$}
\STATE {Choose initial iterate $\signal_0 $}
\FOR{ i = 1 to \texttt{maxiter}}
	\STATE{$\bar \signal_i \gets \signal_{i-1} - s_i \Fo'(\signal_{i-1})^\ast(\Fo (\signal_{i-1})-\data^{\delta})$}
	\COMMENT{gradient step  for $\frac{1}{2}  \norm{ \Fo (\signal)  - \data_{\delta}}^2$}
	\STATE{$\signal_i \gets \bar\signal_i - s_i \alpha \nabla_{\signal} \reg(\Wo, \edot)(\bar \signal_i)$}
	\COMMENT{gradient step  for $\reg(\Wo, \edot)$}
\ENDFOR
\end{algorithmic}
\end{algorithm}
	
In practice, we found that Algorithm~\ref{alg:gradient} gives favorable performance, and is stable with respect to tuning parameters.
Also other algorithms such as proximal gradient methods \cite{combettes2011proximal}
or Newton type methods might be used for the minimization of \eqref{eq:tikU}.
A detailed comparison  with other  algorithms is beyond the scope of this article.

Note that the regularizer  may be taken $ \reg(\Wo,\signal)  =
\norm{\NNe(\signal)}_L$ with an arbitrary norm $\enorm_L$ on $\XXX_L$.
The concrete training  procedure is described  below.
In  the  form  \eqref{eq:tikU}, NETT constitutes   a {\cm non-linear} generalization of $\ell^q$-regularization.

\section{Application to  sparse data tomography}
\label{sec:num}

As a demonstration, we use NETT regularization with the
 encoder-decoder scheme presented in Section~\ref{sec:auto}  to the sparse data problem
in photoacoustic tomography (PAT).
PAT is an emerging hybrid imaging method based on the conversion of light in sound,
and   beneficially  combines the high  contrast of optical imaging with the good resolution of ultrasound tomography (see, for example, \cite{kuchment2008mathematics,paltauf2007photoacoustic,wang2011photoacoustic,Wan09b}).

\subsection{Sparse sampling problem in PAT}
\label{sec:pat}

The  aim of PAT is to recover the initial pressure $\source \colon \R^d \to \R $
in the   wave equation
\begin{equation} \label{eq:wave-fwd}
	\left\{ \begin{aligned}
	&\partial_t^2  p (x,t) - \Delta_x p(x,t)
	=
	0 \,,
	 && \text{ for }
	\kl{x,t} \in
	\R^d \times \kl{0, \infty} \,,
	\\
	&p\kl{x,0}
	=
	\source (x)  \,,
	&& \text{ for }
	x  \in \R^d \,,
	\\
	&\partial_t
	p\kl{x,0}
	=0 \,,
	&& \text{ for }
	x  \in \R^d \,.
\end{aligned} \right.
\end{equation}
form measurements  of $p$ made on an observation surface $S$ outside the
support  of $\source$. Here $d$  is the spatial dimension, $\Delta_x$ the spatial Laplacian, and
$\partial_t$ the derivative with respect to the time variable $t$.
Both cases $d  =2,3 $  for the   spatial dimension are relevant in PAT:
The case $d=3$ corresponds to classical point-wise measurements; the case $d=2$
to integrating  line detectors~\cite{burgholzer2007temporal,paltauf2007photoacoustic}.
In this paper we consider the case of $d=3$ and
 assume the initial pressure $\source \colon \R^2 \to \R $  vanishes outside the unit disc $D_1$,
 the ball of radius $1$, and that acoustic data are collected at the boundary sphere $\sph^1 = \partial D_1$.
In particular, we are interested in  the  sparse sampling  case,
where    data are only given for a small number of sensor
locations on $\sph^1$. This is the case that one often faces in practical applications.

In the full sampling case,  the discrete PAT forward operator is  written
as $\wave \colon \R^{n_1 \times n_2} \to  \R^{m_{\rm full} \times m_2}$ where
$ m_{\rm full}$ corresponds to the number of complete spatial sampling points
and $M_2$ to the number of temporal sampling points.
Sufficient sampling conditions for PAT in the circular geometry have been derived
in \cite{haltmeier2016sampling}. We discretize the  exact inversion formula
of \cite{finch2007inversion} to obtain an approximation $\wave^\sharp   \colon  \R^{m_{\rm full} \times m_2}  \to \R^{n_1 \times n_2}$
to the inverse of $\Fo$.
In the full data case, application of $\wave^\sharp$ to data $\wave \signal \in
\R^{m_{\rm full} \times m_2}$ gives an almost artifact free reconstruction
$\signal \in \R^{n_1 \times n_2}$, see~\cite{haltmeier2016sampling}.
Note that  $\wave^\sharp$ is the discretization  of the continuous adjoint of
$\wave$ with respect  to a weighted  $L^2$-inner product
(see~\cite{finch2007inversion}).

In the sparse sampling case, the PAT forward operator is given by
\begin{equation} \label{eq:pat-s}
	\Fo = \samp \circ \wave \colon \XX =
	\R^{n_1 \times n_2} \to  \R^{m_1 \times m_2}
	\,.
\end{equation}
Here $\samp \colon \R^{m_{\rm full} \times m_2} \to
\R^{m_1 \times m_2}$ is the subsampling operator,  which restricts the  full
data in  to a  small number of spatial  sampling points.
In the case of spatial under-sampling,   the filtered backprojection (FBP)reconstruction
$\Fo^\sharp \coloneqq   \wave^\sharp  \circ \samp^\trans$ yields
typical streak-like under-sampling artifacts
(see, for example, the examples in Figures~\ref{fig:SLphant}).

\begin{figure}[htb!]%
    \includegraphics[width=\textwidth]{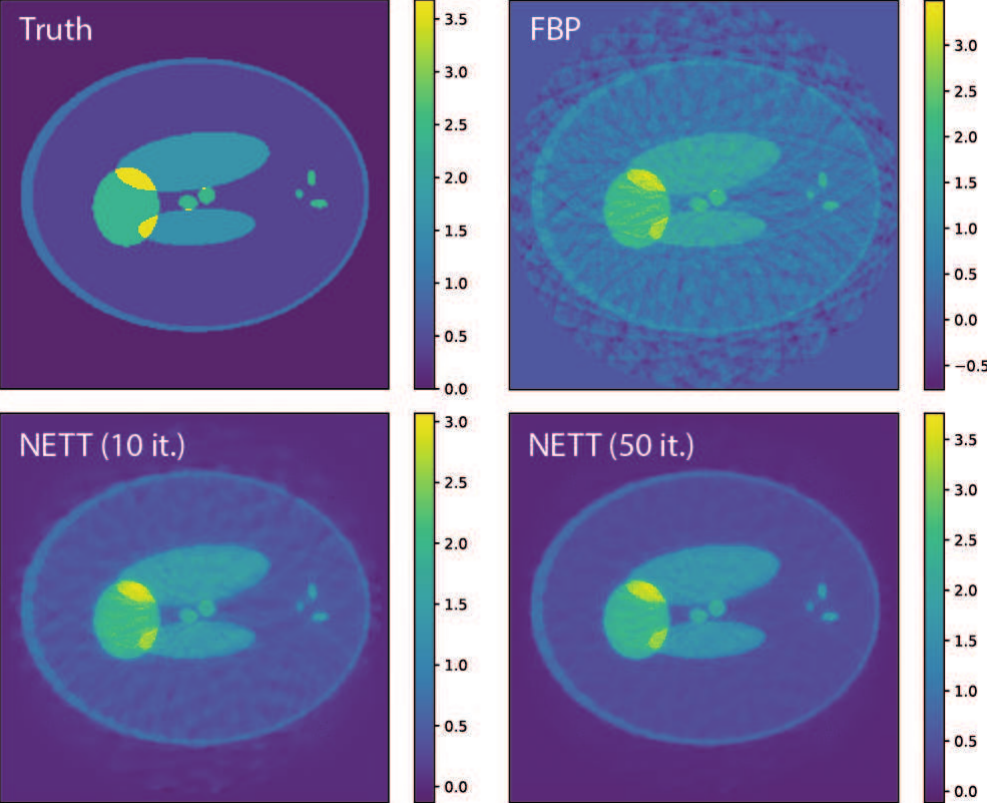}
    \caption{\textbf{Reconstruction results for a  phantom of Shepp-Logan type.}
    Top: \label{fig:SLphant} Phantom $\signal$ (left)
    and corresponding   FBP reconstruction (right);
    Bottom: Iterates $\signal_{10}$ (left) and $\signal_{50}$ (right)
    with  the proposed algorithm for minimizing the NETT functional.}%
\end{figure}

\subsection{Implementation details}
\label{sec:details}

Consider NETT  where the regularizer is defined by the
encoder-decoder framework
described in Section~\ref{sec:auto}.
The network  $\NNd(\WWd,\edot ) \circ \NNe(\WWe,\edot )$  is taken as the
Unet, where the $\NNe(\WWe,\signal)$ corresponds to the output of the
bottom layer with smallest image size and largest depth. The Unet has been proposed in \cite{ronneberger2015unet} for image segmentation and successfully applied  to PAT in  \cite{antholzer2017deep,schwab2018dalnet}.
However, we point out, that
any  network that has the encoder-decoder of the form $\NNd(\WWd,\edot ) \circ \NNe(\WWe,\edot )$ can be used in an analogous manner.

The  network was trained on a set of training
pairs  $\set{(\xxt_\ntrain, \rrt_\ntrain)}_{\ntrain=1}^{\Ntrain_1+\Ntrain_2}$,
with $\Ntrain_1 = \Ntrain_2 = 975$,
 where exactly half of them contained  under-sampling artifacts.
 For generating such training data we used \eqref{eq:z1}, \eqref{eq:z2}   where
 $\zzt_n$ are taken as randomly generated piecewise constant Shepp-Logan type phantoms. The Shepp-Logan type phantoms have position, angle, shape and intensity of every ellipse chosen uniformly at random  under the side constraints that the
 support of every ellipse  lies inside the unit disc and the intensity of the phantom
 is in the range $[0,6]$.
 {\cm During training we had no problems with overfitting or instability and thus we did not use dropout or batch normalization. However, such techniques might be needed for different networks or training sets.}

In this discrete sparse sampling case, we take the forward operator
$\Fo$ as in  \eqref{eq:pat-s} with  $n_1= n_2 = 256$  and $m_1 = 30$
 spatial samples distributed  equidistantly  on the boundary circle.
 We used $m_2 = 2000$ times sampled evenly in the  interval $[0, 2.5]$. The under-sampling problem in PAT
 is solved by FBP, and NETT regularization using $\alpha = 1/4$.   We minimize~\eqref{eq:tikU}
  using  Algorithm~\ref{alg:gradient}, where
  we chose a constant step size of $s_i = 0.4$ and take the zero image
  $\signal_0 = 0$ for the initial guess.
{\cm These parameters have been selected by hand using similar phantoms as reference.}

\begin{figure}[htb!]%
    \includegraphics[width=\textwidth]{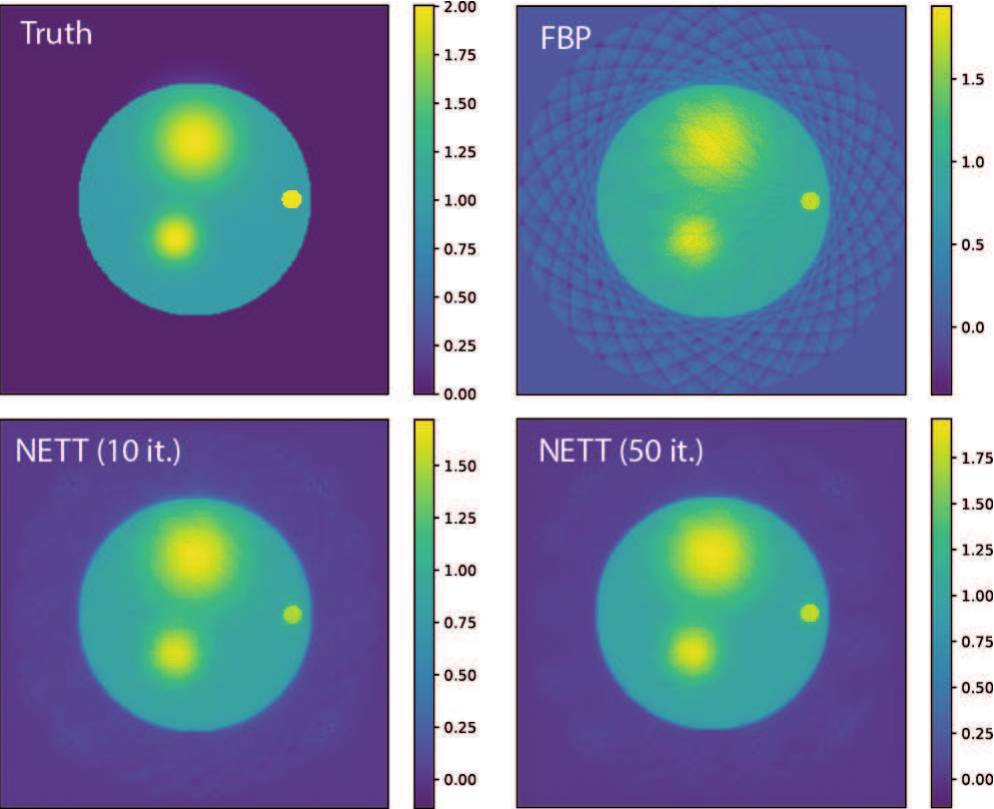}
    \centering
    \caption{\textbf{Reconstruction results for a phantom of different type from
    training  data.}
    Top: \label{fig:SMOOTHphant} Phantom $\signal$
    with smooth blobs (left) and corresponding FBP reconstruction (right);
    Bottom: Iterates  $\signal_{10}$ (left)
    and $\signal_{50}$ (right) with  the proposed algorithm
    for minimizing the NETT functional.}%
\end{figure}

\begin{figure}[htb!]%
    \centering
    \includegraphics[width=\textwidth]{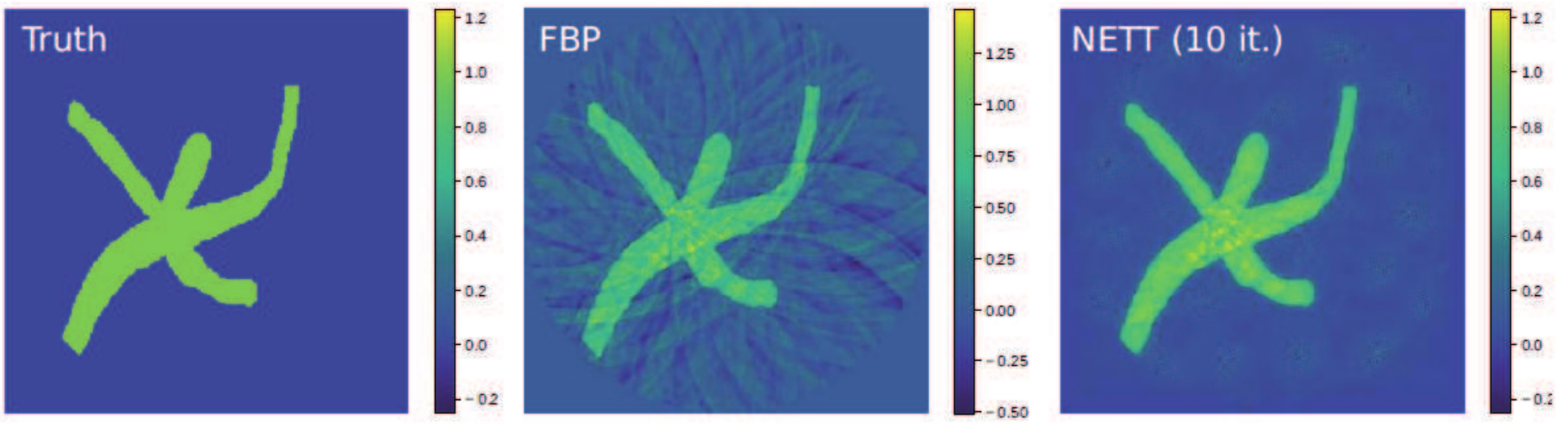}
    \caption{\cm \textbf{Reconstruction results for another phantom of different type from training data.} Left: Phantom $\signal$ with elongated structures. Middle:  Corresponding FBP reconstruction. Right: Result using NETT with 10 iterations.}%
    \label{fig:X_phantom}
\end{figure}

\subsection{Results and discussion}
\label{sec:results}

The top left image in Figure~\ref{fig:SLphant} shows a Shepp-Logan type phantom
$\signal \in \R^{256\times 256}$ corresponding to a function on the domain $[-1,1]^2$.
It is of the same type as the training data, but is not contained in the training data.  The NETT  reconstruction  $\signal_{10}$ and $\signal_{50}$ with Algorithm~\ref{alg:gradient} after  10 and 50 iterations for the Shepp-Logan type
phantom are shown in the bottom row of Figure \ref{fig:SLphant}.
The top right image  Figure \ref{fig:SLphant} shows the reconstruction
$\signal_{\rm FBP} = \Fo^\sharp \Fo  \signal $ with the
FBP algorithm of \cite{finch2007inversion}.
The relative $L^2$-errors $E(\zsignal) \coloneqq \snorm{\signal-\zsignal}_2/ \snorm{\signal}_2$  of the iterates for the Shepp-Logan type phantom
are $E(\signal_{10})=0.262$ and $E(\signal_{50})=0.192$, whereas the relative  error of  the FBP reconstruction
is $E( \signal_{\rm FBP} )=0.338$.
From Figure~\ref{fig:SLphant}  one recognizes that NETT  is able to well remove under-sampling artifacts while preserving high resolution
information.

We  also consider a phantom  image (blobs phantom)
that additionally includes smooth parts and is of different type from the
phantoms used for training.
The blobs phantom  as well as the   FBP reconstruction $\signal_{\rm FBP}$
and  NETT reconstructions $\signal_{10}$ and $\signal_{50}$ are shown in Figure~\ref{fig:SMOOTHphant}.
For the blobs phantom, the relative reconstruction errors are given by $E(\signal_{10})=0.176$,
$E(\signal_{50})=0.102$ and
$E(\signal_{\rm FBP}) = 0.179$.  Again, for this phantom different
from the training set, NETT removes under-sampling artifacts and
at the same time preserves high resolution.
{\cm Results for another phantom of different type from training data are shown in Figure~\ref{fig:X_phantom}.}

\begin{figure}[htb!]%
    \includegraphics[width=\textwidth]{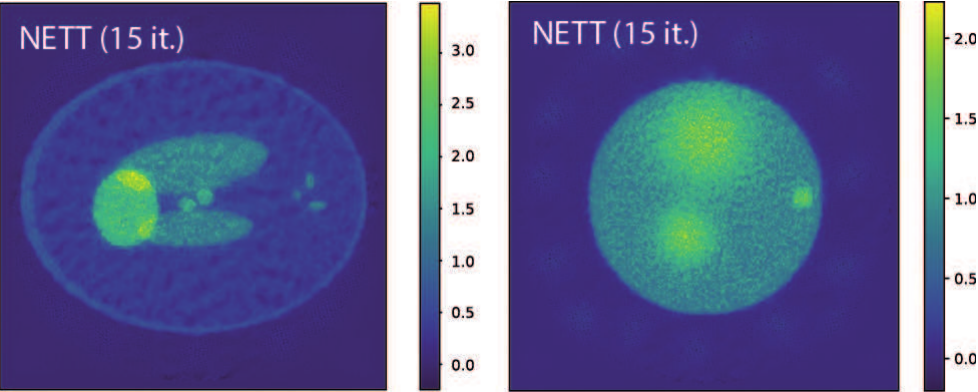}
    \centering
    \caption{\textbf{Reconstruction results from noisy data using
    NETT with 15 iterations.}\label{fig:noisy}
    Left:  Shop Logan phantom;
    Right: Blobs phantom.}%
\end{figure}

Finally, Figure \ref{fig:noisy} shows reconstruction results  with NETT from noisy data
where we added \SI{5}{\percent} additive Gaussian noise to the the data.
We performed 15 iterations with Algorithm~\ref{alg:gradient}. The relative reconstruction errors are
$E(\signal_{15}) = 0.280$ for the Shepp Logan phantom and
$E(\signal_{15}) = 0.210$ for the blobs phantom.
Parameters have been taken  as in the noiseless data case.
{\cm We also calculated the average relative $L^2$-error and average structured similarity index (SSIM) of \cite{wang2004image} on a test set of 100 phantoms, which were similar to the training set. The errors for both the noiseless and noisy case can be seen in Table~\ref{tab:errors}.
We used the same parameters as above for both the noiseless and the noisy case.
}

{\cm
\begin{table}
    \centering
    \begin{tabular}{|c|c|c|}
    &with noise & without noise \\
    \hline 
    relative $L^2$ & 0.20187 &  0.16959 \\
    SSIM &   0.35151 & 0.37066
    \end{tabular}
    \caption{Average errors on the test set of the NETT reconstruction after 20 iterations. For the noisy case we show the error for the iteration with minimal $L^2$ error.}
    \label{tab:errors}
\end{table}
}

In both cases, the  reconstructions are free from under-sampling artifacts and contain high frequency information, which demonstrates  the applicability of NETT for noisy data as well.

The above results demonstrate the   proposed NETT regularization using the encoder-decoder framework  and with Algorithm~\ref{alg:gradient} for minimization  is able to remove under-sampling artifacts. {\cm It also gives consistent} results even on images with smooth structures not contained in the training data. This shows that in the NETT framework, learning the regularization functional on one class of training data, can lead to good results even for images beyond that class.

\section{Conclusion and outlook}
\label{sec:conclusion}

In this paper we developed a new framework for the solution of
inverse problems via NETT \eqref{eq:nett}.
We presented a complete convergence analysis
and  derived   well-posedness and weak convergence
(Theorem~\ref{thm:well}), norm-convergence (Theorem~\ref{thm:strong}),
 as well as  various convergence rates results (see Section~\ref{sec:rates}).
 {\cm For these results we introduced the absolute Bregman distance  as a new generalization  of the   standard Bregman distance from the convex to the non-convex setting.}  NETT combines deep neural networks with a Tikhonov regularization
 strategy.
 The regularizer is defined by a network that might be a  user-specified function
 (generalizing  frame based regularization), or might be a
 CNN trained on an appropriate training data set.
 We have developed a possible strategy for learning a deep CNN (using an encoder-decoder framework, see
 Section~\ref{sec:auto}).
  Initial numerical results for a  sparse data  problem in  PAT  (see Section~\ref{sec:num}) demonstrated
  that  NETT  with the trained   regularizer  works well and also yields good results  for phantoms different from the class of training data. This may be a result of the fact, that  opposed to other deep learning approaches for image reconstruction,  the NETT includes  a data consistency term as well as the trained network that focuses on identifying artifacts.
  Detailed comparison with  other  deep learning methods for inverse problems as well as   variational regularization methods  (including TV-minimization) is subject of  future studies.

Many possible lines of future research arise from  the proposed
NETT regularization  and the corresponding network-minimizing solution concept
\eqref{eq:RminW}. For example, instead of the  Tikhonov variant  \eqref{eq:nett}
 one can  employ
and analyze the residual method (or Ivanov regularization) for approximating \eqref{eq:RminW}, see \cite{GraHalSch11b}.
Instead of the simple incremental gradient descent algorithm
(cf. Algorithm~\ref{alg:gradient}) for minimizing  NETT  one could investigate  different
algorithms  such as proximal gradient  or  semi-smooth Newton methods.
Studying  network designs and training strategies different from the encoder-decoder
scheme is  a promising  aspect of future studies.   Finally, application
of NETT to other inverse  problems is another interesting   research direction.

\section*{Acknowledgement}

SA and MH  acknowledge support of the Austrian Science Fund (FWF), project P 30747.
The work of HL  has been support through the DFG Cluster of Excellence Multiscale Bioimaging EXC 2067.

\end{document}